\documentclass[numbers=enddot,12pt,final,onecolumn,notitlepage]{scrartcl}%
\usepackage[headsepline,footsepline,manualmark]{scrlayer-scrpage}
\usepackage[all,cmtip]{xy}
\usepackage{amssymb}
\usepackage{amsmath}
\usepackage{amsthm}
\usepackage{framed}
\usepackage{comment}
\usepackage{color}
\usepackage[breaklinks=True]{hyperref}
\usepackage[sc]{mathpazo}
\usepackage[T1]{fontenc}
\usepackage{needspace}
\usepackage{tabls}
\usepackage{tikz}
\usepackage{ytableau}
\providecommand{\U}[1]{\protect\rule{.1in}{.1in}}
\newcounter{exer}

\numberwithin{exer}{section}
\theoremstyle{definition}
\newtheorem{theo}{Theorem}[section]
\newenvironment{theorem}[1][]
{\begin{theo}[#1]\begin{leftbar}}
{\end{leftbar}\end{theo}}
\newtheorem{lem}[theo]{Lemma}
\newenvironment{lemma}[1][]
{\begin{lem}[#1]\begin{leftbar}}
{\end{leftbar}\end{lem}}
\newtheorem{prop}[theo]{Proposition}
\newenvironment{proposition}[1][]
{\begin{prop}[#1]\begin{leftbar}}
{\end{leftbar}\end{prop}}
\newtheorem{defi}[theo]{Definition}
\newenvironment{definition}[1][]
{\begin{defi}[#1]\begin{leftbar}}
{\end{leftbar}\end{defi}}
\newtheorem{remk}[theo]{Remark}
\newenvironment{remark}[1][]
{\begin{remk}[#1]\begin{leftbar}}
{\end{leftbar}\end{remk}}
\newtheorem{coro}[theo]{Corollary}

\newtheorem{conv}[theo]{Convention}
\newenvironment{convention}[1][]
{\begin{conv}[#1]\begin{leftbar}}
{\end{leftbar}\end{conv}}
\newtheorem{quest}[theo]{Question}
\newenvironment{question}[1][]
{\begin{quest}[#1]\begin{leftbar}}
{\end{leftbar}\end{quest}}
\newtheorem{warn}[theo]{Warning}
\newenvironment{warning}[1][]
{\begin{warn}[#1]\begin{leftbar}}
{\end{leftbar}\end{warn}}
\newtheorem{conj}[theo]{Conjecture}

\newtheorem{exam}[theo]{Example}
\newenvironment{example}[1][]
{\begin{exam}[#1]\begin{leftbar}}
{\end{leftbar}\end{exam}}
\newtheorem{exmp}[exer]{Exercise}

\let\sumnonlimits\sum
\let\prodnonlimits\prod
\let\cupnonlimits\bigcup
\let\capnonlimits\bigcap
\renewcommand{\sum}{\sumnonlimits\limits}
\renewcommand{\prod}{\prodnonlimits\limits}
\renewcommand{\bigcup}{\cupnonlimits\limits}
\renewcommand{\bigcap}{\capnonlimits\limits}
\excludecomment{verlong}
\includecomment{vershort}
\excludecomment{noncompile}

\newcommand{\arinj}{\ar@{_{(}->}}
\newcommand{\arinjrev}{\ar@{^{(}->}}
\newcommand{\arsurj}{\ar@{->>}}
\newcommand{\arelem}{\ar@{|->}}
\newcommand{\arback}{\ar@{<-}}
\newcommand{\symd}{\mathbin{\bigtriangleup}}

\usetikzlibrary{arrows.meta}
\usetikzlibrary{calc}
\usetikzlibrary{chains}
\usetikzlibrary{shapes}
\usetikzlibrary{decorations.pathmorphing}
\usetikzlibrary{lindenmayersystems}
\definecolor{darkgreen}{rgb}{0,.5,0}
\newtheoremstyle{plainsl}
{8pt plus 2pt minus 4pt}
{8pt plus 2pt minus 4pt}
{\slshape}
{0pt}
{\bfseries}
{.}
{5pt plus 1pt minus 1pt}
{}
\theoremstyle{plainsl}
\ihead{Compositions of $n$-homomorphisms, version 14 April 2026}
\ohead{page \thepage}
\begin{document}

\title{Compositions of $n$-homomorphisms}
\author{Darij Grinberg}
\date{14 April 2026}
\maketitle

\begin{abstract}
\textbf{Abstract.} We study $n$-homomorphisms in the sense of
Khudaverdian--Voronov, but generalized to maps from arbitrary rings to
arbitrary commutative rings. We show that the sum of an $n$-homomorphism and
an $m$-homomorphism is an $\left(  n+m\right)  $-homomorphism, and that the
composition of an $n$-homomorphism and an $m$-homomorphism is an
$nm$-homomorphism. The proofs are entirely combinatorial.

\end{abstract}

This note is a long answer to one of my own MathOverflow questions. The answer
(which was obtained with the help of GPT-5.4, though written up entirely on my
own) did not fit in a MathOverflow post, so I have made it into a note.

\textbf{Acknowledgments.} I have used GPT-5.4 via Christian Stump's convenient
\url{https://sciencebench.ai/} front-end. I thank Christian for his work on this.

\subsection{$n$-homomorphisms}

The notion of an $n$-homomorphism is a generalization of the notion of a
pseudorepresentation, which was introduced by Taylor in \cite{Taylor91} (based
on an idea of Wiles \cite[proof of Lemma 2.2.3]{Wiles88}).

We first introduce some notations. \emph{Rings} are always understood to be
associative and unital (but not necessarily commutative); ring morphisms
preserve the unity (unless we speak of \textquotedblleft nonunital ring
morphisms\textquotedblright). For any nonnegative integer $n$, we let $\left[
n\right]  $ denote the set $\left\{  1,2,\ldots,n\right\}  $, and we let
$S_{n}$ denote the $n$-th symmetric group (i.e., the group of permutations of
this set $\left[  n\right]  $). If $\sigma$ is a permutation of a finite set
(e.g., of $\left[  n\right]  $), then $\left(  -1\right)  ^{\sigma}$ shall
denote the sign of $\sigma$. (If $\sigma\in S_{n}$, then $\left(  -1\right)
^{\sigma}=\left(  -1\right)  ^{\ell\left(  \sigma\right)  }$, where
$\ell\left(  \sigma\right)  $ is the number of inversions of $\sigma$.) The
\emph{cycles} of a permutation $\sigma$ shall be understood in the usual
combinatorial sense; for example, the permutation
\begin{align*}
\left[  7\right]   &  \rightarrow\left[  7\right]  ,\\
i  &  \mapsto8-i
\end{align*}
of $\left[  7\right]  $ has the four cycles $\left(  1,7\right)  $, $\left(
2,6\right)  $, $\left(  3,5\right)  $ and $\left(  4\right)  $. (Fixed points
are $1$-cycles.) Note that cycles are only well-defined up to cyclic rotation;
for example, $\left(  2,6\right)  $ and $\left(  6,2\right)  $ are the same cycle.

In \cite{BucRee97}, Buchstaber and Rees defined the notion of an
$n$\emph{-homomorphism} between commutative rings; it was studied further in
\cite{BucRee01}, \cite{BucRee04} and \cite{KhuVor20} (among other works),
although mostly only for commutative $\mathbb{Q}$-algebras. We shall here
generalize it to a notion of $n$-homomorphisms from an arbitrary (not
necessarily commutative) ring $A$ to a commutative ring $B$. We define it as follows:

\begin{definition}
\label{def.1}Let $A$ be a ring, and let $B$ be a commutative ring. Let
$f:A\rightarrow B$ be any $\mathbb{Z}$-linear map.

\begin{enumerate}
\item[\textbf{(a)}] We say that $f$ is \emph{central} if we have
\[
f\left(  aa^{\prime}\right)  =f\left(  a^{\prime}a\right)  \qquad\text{for all
}a,a^{\prime}\in A.
\]

\item[\textbf{(b)}] If $f$ is central, and if $n$ is a nonnegative integer,
then we define the $\mathbb{Z}$-multilinear map $f_{n}:A^{n}\rightarrow B$ by
\begin{align}
&  f_{n}\left(  a_{1},a_{2},\ldots,a_{n}\right) \nonumber\\
&  =\sum_{\sigma\in S_{n}}\left(  -1\right)  ^{\sigma}\prod
_{\substack{c=\left(  i_{1},i_{2},\ldots,i_{k}\right)  \\\text{is a cycle of
}\sigma}}f\left(  a_{i_{1}}a_{i_{2}}\cdots a_{i_{k}}\right)
\label{eq.darij1.1}%
\end{align}
(we will see in a moment why this is well-defined). For example, for $n=3$,
this is saying that%
\begin{align*}
f_{3}\left(  a_{1},a_{2},a_{3}\right)   &  =f\left(  a_{1}\right)  f\left(
a_{2}\right)  f\left(  a_{3}\right)  -f\left(  a_{1}\right)  f\left(
a_{2}a_{3}\right)  -f\left(  a_{2}\right)  f\left(  a_{1}a_{3}\right) \\
&  \qquad-f\left(  a_{3}\right)  f\left(  a_{1}a_{2}\right)  +f\left(
a_{1}a_{2}a_{3}\right)  +f\left(  a_{1}a_{3}a_{2}\right)
\end{align*}
(here, the first addend corresponds to the permutation $\sigma
=\operatorname*{id}\in S_{3}$ with its three cycles $\left(  1\right)
,\left(  2\right)  ,\left(  3\right)  $; the second addend corresponds to the
transposition $t_{2,3}\in S_{3}$ that swaps $2$ and $3$ and has two cycles
$\left(  1\right)  $ and $\left(  2,3\right)  $; the following two addends
similarly correspond to the transpositions $t_{1,3}$ and $t_{1,2}$; and the
last two addends correspond to the two $3$-cycles in $S_{3}$).

Note that the right hand side of \eqref{eq.darij1.1} is well-defined, since
the centrality of $f$ ensures that the value $f\left(  a_{i_{1}}a_{i_{2}%
}\cdots a_{i_{k}}\right)  $ does not depend on where we start indexing the
cycle $c$ (indeed, if we replace the cycle $\left(  i_{1},i_{2},\ldots
,i_{k}\right)  $ by its cyclic rotation $\left(  i_{p},i_{p+1},\ldots
,i_{k},i_{1},i_{2},\ldots,i_{p-1}\right)  $, then the centrality of $f$ shows
that $f\left(  a_{i_{1}}a_{i_{2}}\cdots a_{i_{k}}\right)  =f\left(  a_{i_{p}%
}a_{i_{p+1}}\cdots a_{i_{k}}a_{i_{1}}a_{i_{2}}\cdots a_{i_{p-1}}\right)  $).

The map $f_{n}$ is called the $n$\emph{-Frobenius map} of $f$. For instance,
for all $a,b,c\in A$, we have
\begin{align*}
f_{0}\left(  {}\right)   &  =1;\\
f_{1}\left(  a\right)   &  =f\left(  a\right)  ;\\
f_{2}\left(  a,b\right)   &  =f\left(  a\right)  f\left(  b\right)  -f\left(
ab\right)  ;\\
f_{3}\left(  a,b,c\right)   &  =f\left(  a\right)  f\left(  b\right)  f\left(
c\right)  -f\left(  a\right)  f\left(  bc\right) \\
&  \qquad-f\left(  b\right)  f\left(  ac\right)  -f\left(  c\right)  f\left(
ab\right)  +f\left(  abc\right)  +f\left(  acb\right)  .
\end{align*}
As we will soon see (Proposition \ref{prop.2}), the map $f_{n}$ can also be
defined recursively by
\begin{align}
&  f_{n}\left(  a_{1},a_{2},\ldots,a_{n}\right) \nonumber\\
&  =f\left(  a_{n}\right)  f_{n-1}\left(  a_{1},a_{2},\ldots,a_{n-1}\right)
\nonumber\\
&  \qquad-\sum_{i=1}^{n-1}f_{n-1}\left(  a_{1},a_{2},\ldots,a_{i-1},a_{i}%
a_{n},a_{i+1},a_{i+2},\ldots,a_{n-1}\right)
.\ \ \ \ \ \ \ \ \ \ \label{eq.darij1.2}%
\end{align}

\item[\textbf{(c)}] Let $n$ be a nonnegative integer. We say that $f$ is an
$n$\emph{-homomorphism} (or \emph{Frobenius }$n$\emph{-homomorphism}) if $f$
is central and satisfies $f_{n+1}=0$ (identically).
\end{enumerate}
\end{definition}

\begin{example}
The only $0$-homomorphism is the zero map. The $1$-homomorphisms are just the
nonunital ring homomorphisms. The $2$-homomorphisms are the central maps
$f:A\rightarrow B$ that satisfy%
\begin{align*}
&  f\left(  a\right)  f\left(  b\right)  f\left(  c\right)  -f\left(
a\right)  f\left(  bc\right)  -f\left(  b\right)  f\left(  ac\right)
-f\left(  c\right)  f\left(  ab\right)  +f\left(  abc\right)  +f\left(
acb\right) \\
&  =0\qquad\text{for all }a,b,c\in A.
\end{align*}

An example of an $n$-homomorphism is the trace map $\operatorname*{Tr}%
:B^{n\times n}\rightarrow B$ from the matrix ring $B^{n\times n}$ (sending
each $n\times n$-matrix to its trace). Indeed, the fact that
$\operatorname*{Tr}\nolimits_{n+1}=0$ on $B^{n\times n}$ is known as the
\emph{fundamental trace identity} for $n\times n$-matrices, and goes back to
Frobenius; proofs can be found (e.g.) in \cite[Corollary]{Laue87},
\cite[Theorem 1]{Dotsen11}, \cite[\S VII.7.3.2]{Morel19} or \cite[Theorem 4.3
(b)]{Proces76}.
\end{example}

This general concept of an $n$-homomorphism also covers another classical
notion: that of a pseudocharacter, as considered (e.g.) in \cite{Dotsen11},
\cite[\S VII.7.3]{Morel19}, \cite[Definition 2.2]{Bellai10}, \cite{Rouqui96},
\cite{Chenev08}, generalizing the pseudorepresentations introduced by Taylor
in \cite{Taylor91} (which, in turn, were based on an idea of Wiles \cite[proof
of Lemma 2.2.3]{Wiles88}). Again, much of the literature requires that $A$ is
a $\mathbb{Q}$-algebra or at least $n!$ is invertible. If we turn a blind eye
to this requirement, then $n$-homomorphisms are more or less the same as
pseudocharacters: A $\mathbb{Z}$-linear map $f:A\rightarrow B$ from a ring $A$
to a commutative ring $B$ is an $n$-homomorphism if and only if the
corresponding $B$-linear map
\begin{align*}
\widetilde{f}:B\otimes_{\mathbb{Z}}A  &  \rightarrow B,\\
b\otimes a  &  \mapsto bf\left(  a\right)
\end{align*}
is a pseudocharacter of degree $n$. Conversely, a $B$-linear map
$f:A\rightarrow B$ from a $B$-algebra $A$ to its (commutative) base ring $B$
is a pseudocharacter of degree $n$ if and only if it is an $n$-homomorphism.
Thus, the concept of an $n$-homomorphism and that of a pseudocharacter of
degree $n$ subsume each other (again ignoring the invertibility requirement).
Apparently the two respective communities (pseudocharacters and $n$%
-homomorphisms) are mostly unaware of one another, even though both credit
Frobenius for the original idea.

In this note, I will prove certain basic (but nontrivial) properties of
$n$-homomorphisms in full generality, most importantly \cite[Theorem
3.2]{KhuVor20}, without assuming $A$ commutative or $n!$ invertible. The
method used in \cite{KhuVor20} becomes unusable in this generality, so one is
forced to do combinatorics.

The first main result of this note is a generalization of the first part of
\cite[Theorem 3.2]{KhuVor20} (also part of \cite[Lemme 2.8]{Rouqui96}%
):\footnote{The notation $\mathbb{N}$ denotes the set of all nonnegative
integers (including $0$).}

\begin{theorem}
\label{thm.1}Let $A$ and $B$ be two rings, with $B$ commutative. Let
$n,m\in\mathbb{N}$. Let $f:A\rightarrow B$ be an $n$-homomorphism, and let
$g:A\rightarrow B$ be an $m$-homomorphism. Then, $f+g$ is an $\left(
n+m\right)  $-homomorphism.
\end{theorem}

In short: the sum of an $n$-homomorphism with an $m$-homomorphism is an
$\left(  n+m\right)  $-homomorphism.

The second main result generalizes the second part of \cite[Theorem
3.2]{KhuVor20}:

\begin{theorem}
\label{thm.6}Let $A$, $B$ and $C$ be three rings, with $B$ and $C$
commutative. Let $n,m\in\mathbb{N}$. Let $f:B\rightarrow C$ be an
$n$-homomorphism, and let $g:A\rightarrow B$ be an $m$-homomorphism. Then,
$f\circ g$ is an $nm$-homomorphism.
\end{theorem}

In short: the composition of an $n$-homomorphism with an $m$-homomorphism is
an $nm$-homomorphism.

The proofs will use some auxiliary results that may well be useful on their
own. In particular, Theorem \ref{thm.4} is a formula for the $n$-Frobenius
maps of a composition of two central maps; this formula was found by the
GPT-5.4 LLM when I asked it for a proof of Theorem \ref{thm.6}.

\begin{remark}
It is worth mentioning the \textquotedblleft aspirational\textquotedblright%
\ properties of $n$-homomorphisms and pseudocharacters, as they can be
helpful. A pseudocharacter of degree $n$ wishes to be the character of an
$n$-dimensional representation of the algebra (in the sense that the latter
are always instances of the former, and in some sense \textquotedblleft
generic\textquotedblright\ ones; see, e.g., \cite[Proposition 2]{Dotsen11}).
An $n$-homomorphism wishes to be a sum of $n$ ring homomorphisms (see
\cite[\S 1.2]{KhuVor20}). The former wish can be fulfilled under certain
restrictive conditions (see, e.g., \cite[\S VII.7.3.4]{Morel19}). I don't know
when and to what extent the latter can be fulfilled. However, when it can be
fulfilled, Theorem \ref{thm.1} becomes obvious, and Theorem \ref{thm.6} also
becomes easy (if $f:B\rightarrow C$ is a sum of $n$ (nonunital) ring
homomorphisms, and if $g$ is an $m$-homomorphism, then $f\circ g$ is an
$nm$-homomorphism, as can be easily proved using Theorem \ref{thm.1}). Sadly,
this approach seems to be unsuited for the generality in which we have stated
the above theorems.
\end{remark}

\subsection{Recursion and explicit formula}

First, I will pay a debt from Definition \ref{def.1} \textbf{(b)}, by proving
the equivalence of the two definitions of $n$-Frobenius maps:

\begin{proposition}
\label{prop.2}Let $A$ be a ring, and let $B$ be a commutative ring. Let
$f:A\rightarrow B$ be a central $\mathbb{Z}$-linear map.

The recursive definition \eqref{eq.darij1.2} of $f_{n}$ (with the base case
$f_{0}\left(  {}\right)  :=1$) is equivalent to the explicit definition \eqref{eq.darij1.1}.
\end{proposition}

\begin{proof}
[Proof sketch.]We use \eqref{eq.darij1.1} as the definition of $f_{n}$. Thus,
we must prove that $f_{n}$ satisfies the recursion \eqref{eq.darij1.2}.

Let $n\geq1$ and $a_{1},a_{2},\ldots,a_{n}\in A$. Then, \eqref{eq.darij1.1}
becomes
\begin{align}
&  f_{n}\left(  a_{1},a_{2},\ldots,a_{n}\right) \nonumber\\
&  =\sum_{\sigma\in S_{n}}\left(  -1\right)  ^{\sigma}\prod
_{\substack{c=\left(  i_{1},i_{2},\ldots,i_{k}\right)  \\\text{is a cycle of
}\sigma}}f\left(  a_{i_{1}}a_{i_{2}}\cdots a_{i_{k}}\right) \nonumber\\
&  =\sum_{i=1}^{n}\ \ \sum_{\substack{\sigma\in S_{n};\\\sigma\left(
i\right)  =n}}\left(  -1\right)  ^{\sigma}\prod_{\substack{c=\left(
i_{1},i_{2},\ldots,i_{k}\right)  \\\text{is a cycle of }\sigma}}f\left(
a_{i_{1}}a_{i_{2}}\cdots a_{i_{k}}\right)  \label{eq.darij2.p2.pf.1}%
\end{align}
(since each $\sigma\in S_{n}$ satisfies $\sigma\left(  i\right)  =n$ for a
unique element $i\in\left\{  1,2,\ldots,n\right\}  $). We shall now rewrite
the addends of the outer sum here, showing that they correspond (up to sign)
to the addends on the right hand side of \eqref{eq.darij1.2}.

We start with the addend for $i=n$.

Each permutation $\tau\in S_{n-1}$ can be extended to a permutation
$\widehat{\tau}\in S_{n}$ by setting $\widehat{\tau}\left(  n\right)  :=n$ and
$\widehat{\tau}\left(  i\right)  :=\tau\left(  i\right)  $ for all $i<n$. The
assignment $\tau\mapsto\widehat{\tau}$ defines a bijection from $S_{n-1}$ to
$\left\{  \sigma\in S_{n}\ \mid\ \sigma\left(  n\right)  =n\right\}  $. Let us
use this bijection to reindex a sum:
\begin{align}
&  \sum_{\substack{\sigma\in S_{n};\\\sigma\left(  n\right)  =n}}\left(
-1\right)  ^{\sigma}\prod_{\substack{c=\left(  i_{1},i_{2},\ldots
,i_{k}\right)  \\\text{is a cycle of }\sigma}}f\left(  a_{i_{1}}a_{i_{2}%
}\cdots a_{i_{k}}\right) \nonumber\\
&  =\sum_{\tau\in S_{n-1}}\underbrace{\left(  -1\right)  ^{\widehat{\tau}}%
}_{\substack{=\left(  -1\right)  ^{\tau}\\\text{(since }\widehat{\tau}\text{
has the same}\\\text{inversions as }\tau\text{)}}}\ \ \underbrace{\prod
_{\substack{c=\left(  i_{1},i_{2},\ldots,i_{k}\right)  \\\text{is a cycle of
}\widehat{\tau}}}f\left(  a_{i_{1}}a_{i_{2}}\cdots a_{i_{k}}\right)
}_{\substack{=f\left(  a_{n}\right)  \prod_{\substack{c=\left(  i_{1}%
,i_{2},\ldots,i_{k}\right)  \\\text{is a cycle of }\tau}}f\left(  a_{i_{1}%
}a_{i_{2}}\cdots a_{i_{k}}\right)  \\\text{(since the cycles of }%
\widehat{\tau}\text{ are the cycles of }\tau\\\text{plus the additional
}1\text{-cycle }\left(  n\right)  \text{)}}}\nonumber\\
&  =\sum_{\tau\in S_{n-1}}\left(  -1\right)  ^{\tau}f\left(  a_{n}\right)
\prod_{\substack{c=\left(  i_{1},i_{2},\ldots,i_{k}\right)  \\\text{is a cycle
of }\tau}}f\left(  a_{i_{1}}a_{i_{2}}\cdots a_{i_{k}}\right) \nonumber\\
&  =f\left(  a_{n}\right)  \underbrace{\sum_{\tau\in S_{n-1}}\left(
-1\right)  ^{\tau}\prod_{\substack{c=\left(  i_{1},i_{2},\ldots,i_{k}\right)
\\\text{is a cycle of }\tau}}f\left(  a_{i_{1}}a_{i_{2}}\cdots a_{i_{k}%
}\right)  }_{\substack{=f_{n-1}\left(  a_{1},a_{2},\ldots,a_{n-1}\right)
\\\text{(by the definition of }f_{n-1}\text{)}}}\nonumber\\
&  =f\left(  a_{n}\right)  f_{n-1}\left(  a_{1},a_{2},\ldots,a_{n-1}\right)  .
\label{eq.darij2.p2.pf.2}%
\end{align}

Now, let $i\in\left\{  1,2,\ldots,n-1\right\}  $. Let us define $n-1$ elements
$a_{1}^{\prime},a_{2}^{\prime},\ldots,a_{n-1}^{\prime}$ of $A$ by
\begin{align}
&  \left(  a_{1}^{\prime},a_{2}^{\prime},\ldots,a_{n-1}^{\prime}\right)
\nonumber\\
&  :=\left(  a_{1},a_{2},\ldots,a_{i-1},a_{i}a_{n},a_{i+1},a_{i+2}%
,\ldots,a_{n-1}\right)  . \label{eq.darij2.p2.pf.5a}%
\end{align}
That is, we set $a_{i}^{\prime}:=a_{i}a_{n}$ and
\begin{equation}
a_{r}^{\prime}:=a_{r}\qquad\text{for all }r\neq i. \label{eq.darij2.p2.pf.5}%
\end{equation}

Let $t_{i,n}\in S_{n}$ be the transposition that swaps $i$ with $n$.

Let $\tau\in S_{n-1}$ be any permutation. Then, the permutation $\widehat{\tau
}\in S_{n}$ (defined above) sends $n$ to $n$. Hence, the permutation
$\widehat{\tau}\circ t_{i,n}\in S_{n}$ sends $i\overset{t_{i,n}}{\mapsto
}n\overset{\widehat{\tau}}{\mapsto}n$. Moreover, this permutation
$\widehat{\tau}\circ t_{i,n}$ has \textquotedblleft almost\textquotedblright%
\ the same cycles as $\tau$: Namely, let
\[
d=\left(  j_{1},j_{2},\ldots,j_{\ell}\right)
\]
be the cycle of $\tau$ that contains $i$, indexed in such a way that
$i=j_{\ell}$. (If $\tau$ fixes $i$, then this is a $1$-cycle, i.e., we have
$\ell=1$.) Then, the cycles of $\widehat{\tau}\circ t_{i,n}$ are exactly the
cycles of $\tau$, except that the cycle $d$ is replaced by
\[
d^{\prime}:=\left(  j_{1},j_{2},\ldots,j_{\ell},n\right)  .
\]
(To see this, just observe that $\widehat{\tau}\circ t_{i,n}$ transforms the
inputs $j_{\ell}$ and $n$ as follows: $j_{\ell}=i\overset{t_{i,n}}{\mapsto
}n\overset{\widehat{\tau}}{\mapsto}n$ and $n\overset{t_{i,n}}{\mapsto
}i=j_{\ell}\overset{\widehat{\tau}}{\mapsto}j_{1}$. On all other inputs,
$\widehat{\tau}\circ t_{i,n}$ does not differ from $\tau$, since $t_{i,n}$
only affects the inputs $i$ and $n$.) In other words, the cycles of
$\widehat{\tau}\circ t_{i,n}$ are exactly the cycle $d^{\prime}=\left(
j_{1},j_{2},\ldots,j_{\ell},n\right)  $ and the cycles of $\tau$ distinct from
$d$. Therefore,
\begin{align}
&  \prod_{\substack{c=\left(  i_{1},i_{2},\ldots,i_{k}\right)  \\\text{is a
cycle of }\widehat{\tau}\circ t_{i,n}}}f\left(  a_{i_{1}}a_{i_{2}}\cdots
a_{i_{k}}\right) \nonumber\\
&  =f\left(  a_{j_{1}}a_{j_{2}}\cdots a_{j_{\ell}}a_{n}\right)  \cdot
\prod_{\substack{c=\left(  i_{1},i_{2},\ldots,i_{k}\right)  \\\text{is a cycle
of }\tau\\\text{distinct from }d}}f\left(  a_{i_{1}}a_{i_{2}}\cdots a_{i_{k}%
}\right)  . \label{eq.darij2.p2.pf.7}%
\end{align}

However, for each cycle $c=\left(  i_{1},i_{2},\ldots,i_{k}\right)  $ of
$\tau$ distinct from $d$, we have
\begin{equation}
a_{i_{1}}a_{i_{2}}\cdots a_{i_{k}}=a_{i_{1}}^{\prime}a_{i_{2}}^{\prime}\cdots
a_{i_{k}}^{\prime} \label{eq.darij2.p2.pf.8a}%
\end{equation}
(because $c\neq d$ ensures that $c$ does not contain $i$, and therefore all
elements $i_{1},i_{2},\ldots,i_{k}$ of $c$ are distinct from $i$ and thus
satisfy $a_{i_{1}}^{\prime}=a_{i_{1}}$ and $a_{i_{2}}^{\prime}=a_{i_{2}}$ and
so on (by \eqref{eq.darij2.p2.pf.5}); hence, $a_{i_{1}}^{\prime}a_{i_{2}%
}^{\prime}\cdots a_{i_{k}}^{\prime}=a_{i_{1}}a_{i_{2}}\cdots a_{i_{k}}$).
Meanwhile, the cycle $d=\left(  j_{1},j_{2},\ldots,j_{\ell}\right)  $ of
$\tau$ contains $i$ only as its last entry $j_{\ell}=i$, and therefore all the
previous elements $j_{1},j_{2},\ldots,j_{\ell-1}$ of $d$ are distinct from $i$
and thus satisfy $a_{j_{1}}^{\prime}=a_{j_{1}}$ and $a_{j_{2}}^{\prime
}=a_{j_{2}}$ and so on (by \eqref{eq.darij2.p2.pf.5}), whereas its last entry
satisfies $a_{j_{\ell}}^{\prime}=a_{i}^{\prime}=a_{i}a_{n}=a_{j_{\ell}}a_{n}$
(since $i=j_{\ell}$). Therefore,
\[
a_{j_{1}}^{\prime}a_{j_{2}}^{\prime}\cdots a_{j_{\ell-1}}^{\prime}a_{j_{\ell}%
}^{\prime}=a_{j_{1}}a_{j_{2}}\cdots a_{j_{\ell-1}}a_{j_{\ell}}a_{n}=a_{j_{1}%
}a_{j_{2}}\cdots a_{j_{\ell}}a_{n}.
\]
In other words,
\begin{equation}
a_{j_{1}}a_{j_{2}}\cdots a_{j_{\ell}}a_{n}=a_{j_{1}}^{\prime}a_{j_{2}}%
^{\prime}\cdots a_{j_{\ell-1}}^{\prime}a_{j_{\ell}}^{\prime}=a_{j_{1}}%
^{\prime}a_{j_{2}}^{\prime}\cdots a_{j_{\ell}}^{\prime}.
\label{eq.darij2.p2.pf.8}%
\end{equation}
Using \eqref{eq.darij2.p2.pf.8} and \eqref{eq.darij2.p2.pf.8a}, we can rewrite
\eqref{eq.darij2.p2.pf.7} as
\begin{align}
&  \prod_{\substack{c=\left(  i_{1},i_{2},\ldots,i_{k}\right)  \\\text{is a
cycle of }\widehat{\tau}\circ t_{i,n}}}f\left(  a_{i_{1}}a_{i_{2}}\cdots
a_{i_{k}}\right) \nonumber\\
&  =f\left(  a_{j_{1}}^{\prime}a_{j_{2}}^{\prime}\cdots a_{j_{\ell}}^{\prime
}\right)  \cdot\prod_{\substack{c=\left(  i_{1},i_{2},\ldots,i_{k}\right)
\\\text{is a cycle of }\tau\\\text{distinct from }d}}f\left(  a_{i_{1}%
}^{\prime}a_{i_{2}}^{\prime}\cdots a_{i_{k}}^{\prime}\right) \nonumber\\
&  =\prod_{\substack{c=\left(  i_{1},i_{2},\ldots,i_{k}\right)  \\\text{is a
cycle of }\tau}}f\left(  a_{i_{1}}^{\prime}a_{i_{2}}^{\prime}\cdots a_{i_{k}%
}^{\prime}\right)  . \label{eq.darij2.p2.pf.9}%
\end{align}
Finally,
\begin{align}
\left(  -1\right)  ^{\widehat{\tau}\circ t_{i,n}}  &  =\underbrace{\left(
-1\right)  ^{\widehat{\tau}}}_{\substack{=\left(  -1\right)  ^{\tau
}\\\text{(since }\widehat{\tau}\text{ has the same}\\\text{inversions as }%
\tau\text{)}}}\ \ \underbrace{\left(  -1\right)  ^{t_{i,n}}}%
_{\substack{=-1\\\text{(since transpositions}\\\text{have sign }-1\text{)}%
}}\nonumber\\
&  =-\left(  -1\right)  ^{\tau}. \label{eq.darij2.p2.pf.10}%
\end{align}

Forget that we fixed $\tau$. Thus, for each permutation $\tau\in S_{n-1}$, we
have constructed a permutation $\widehat{\tau}\circ t_{i,n}\in S_{n}$ that
sends $i$ to $n$ and satisfies \eqref{eq.darij2.p2.pf.9} and
\eqref{eq.darij2.p2.pf.10}. Moreover, the assignment $\tau\mapsto
\widehat{\tau}\circ t_{i,n}$ defines a bijection from $S_{n-1}$ to $\left\{
\sigma\in S_{n}\ \mid\ \sigma\left(  i\right)  =n\right\}  $ (because if
$\sigma\in S_{n}$ sends $i$ to $n$, then $\sigma\circ t_{i,n}^{-1}$ sends $n$
to $n$ and thus has the form $\widehat{\tau}$ for a unique $\tau\in S_{n-1}$).
Let us use this bijection to reindex a sum:
\begin{align}
&  \sum_{\substack{\sigma\in S_{n};\\\sigma\left(  i\right)  =n}}\left(
-1\right)  ^{\sigma}\prod_{\substack{c=\left(  i_{1},i_{2},\ldots
,i_{k}\right)  \\\text{is a cycle of }\sigma}}f\left(  a_{i_{1}}a_{i_{2}%
}\cdots a_{i_{k}}\right) \nonumber\\
&  =\sum_{\tau\in S_{n-1}}\underbrace{\left(  -1\right)  ^{\widehat{\tau}\circ
t_{i,n}}}_{\substack{=-\left(  -1\right)  ^{\tau}\\\text{(by
\eqref{eq.darij2.p2.pf.10})}}}\ \ \underbrace{\prod_{\substack{c=\left(
i_{1},i_{2},\ldots,i_{k}\right)  \\\text{is a cycle of }\widehat{\tau}\circ
t_{i,n}}}f\left(  a_{i_{1}}a_{i_{2}}\cdots a_{i_{k}}\right)  }%
_{\substack{=\prod_{\substack{c=\left(  i_{1},i_{2},\ldots,i_{k}\right)
\\\text{is a cycle of }\tau}}f\left(  a_{i_{1}}^{\prime}a_{i_{2}}^{\prime
}\cdots a_{i_{k}}^{\prime}\right)  \\\text{(by \eqref{eq.darij2.p2.pf.9})}%
}}\nonumber\\
&  =-\underbrace{\sum_{\tau\in S_{n-1}}\left(  -1\right)  ^{\tau}%
\prod_{\substack{c=\left(  i_{1},i_{2},\ldots,i_{k}\right)  \\\text{is a cycle
of }\tau}}f\left(  a_{i_{1}}^{\prime}a_{i_{2}}^{\prime}\cdots a_{i_{k}%
}^{\prime}\right)  }_{\substack{=f_{n-1}\left(  a_{1}^{\prime},a_{2}^{\prime
},\ldots,a_{n-1}^{\prime}\right)  \\\text{(by the definition of }%
f_{n-1}\text{)}}}\nonumber\\
&  =-f_{n-1}\left(  a_{1}^{\prime},a_{2}^{\prime},\ldots,a_{n-1}^{\prime
}\right) \nonumber\\
&  =-f_{n-1}\left(  a_{1},a_{2},\ldots,a_{i-1},a_{i}a_{n},a_{i+1}%
,a_{i+2},\ldots,a_{n-1}\right)  \label{eq.darij2.p2.pf.15}%
\end{align}
(by \eqref{eq.darij2.p2.pf.5a}).

Forget that we fixed $i$. So we have proved \eqref{eq.darij2.p2.pf.15} for
each $i\in\left\{  1,2,\ldots,n-1\right\}  $. Now, splitting off the $i=n$
addend from the outer sum on the right hand side of \eqref{eq.darij2.p2.pf.1},
we obtain
\begin{align*}
f_{n}\left(  a_{1},a_{2},\ldots,a_{n}\right)   &  =\underbrace{\sum
_{\substack{\sigma\in S_{n};\\\sigma\left(  n\right)  =n}}\left(  -1\right)
^{\sigma}\prod_{\substack{c=\left(  i_{1},i_{2},\ldots,i_{k}\right)
\\\text{is a cycle of }\sigma}}f\left(  a_{i_{1}}a_{i_{2}}\cdots a_{i_{k}%
}\right)  }_{\substack{=f\left(  a_{n}\right)  f_{n-1}\left(  a_{1}%
,a_{2},\ldots,a_{n-1}\right)  \\\text{(by \eqref{eq.darij2.p2.pf.2})}}}\\
&  \qquad+\sum_{i=1}^{n-1}\ \ \underbrace{\sum_{\substack{\sigma\in
S_{n};\\\sigma\left(  i\right)  =n}}\left(  -1\right)  ^{\sigma}%
\prod_{\substack{c=\left(  i_{1},i_{2},\ldots,i_{k}\right)  \\\text{is a cycle
of }\sigma}}f\left(  a_{i_{1}}a_{i_{2}}\cdots a_{i_{k}}\right)  }%
_{\substack{=-f_{n-1}\left(  a_{1},a_{2},\ldots,a_{i-1},a_{i}a_{n}%
,a_{i+1},a_{i+2},\ldots,a_{n-1}\right)  \\\text{(by
\eqref{eq.darij2.p2.pf.15})}}}\\
&  =f\left(  a_{n}\right)  f_{n-1}\left(  a_{1},a_{2},\ldots,a_{n-1}\right) \\
&  \qquad-\sum_{i=1}^{n-1}f_{n-1}\left(  a_{1},a_{2},\ldots,a_{i-1},a_{i}%
a_{n},a_{i+1},a_{i+2},\ldots,a_{n-1}\right)  .
\end{align*}
This proves \eqref{eq.darij1.2}. Thus, Proposition \ref{prop.2} is proved.
\end{proof}

\subsection{Symmetry of $n$-Frobenius maps}

Next, we note something simple:

\begin{proposition}
\label{prop.3}Let $A$ be a ring, and let $B$ be a commutative ring. Let
$f:A\rightarrow B$ be a central $\mathbb{Z}$-linear map. Let $n\in\mathbb{N}$.
Then, the map $f_{n}:A^{n}\rightarrow B$ is symmetric in its $n$ inputs. That
is, if $\tau\in S_{n}$ is any permutation and $a_{1},a_{2},\ldots,a_{n}\in A$,
then
\[
f_{n}\left(  a_{\tau\left(  1\right)  },a_{\tau\left(  2\right)  }%
,\ldots,a_{\tau\left(  n\right)  }\right)  =f_{n}\left(  a_{1},a_{2}%
,\ldots,a_{n}\right)  .
\]

\end{proposition}

\begin{proof}
Let $\tau\in S_{n}$ and $a_{1},a_{2},\ldots,a_{n}\in A$. Then,
\eqref{eq.darij1.1} yields
\begin{align}
&  f_{n}\left(  a_{1},a_{2},\ldots,a_{n}\right) \nonumber\\
&  =\sum_{\sigma\in S_{n}}\left(  -1\right)  ^{\sigma}\prod
_{\substack{c=\left(  i_{1},i_{2},\ldots,i_{k}\right)  \\\text{is a cycle of
}\sigma}}f\left(  a_{i_{1}}a_{i_{2}}\cdots a_{i_{k}}\right) \nonumber\\
&  =\sum_{\sigma\in S_{n}}\left(  -1\right)  ^{\sigma}\prod
_{\substack{c=\left(  j_{1},j_{2},\ldots,j_{k}\right)  \\\text{is a cycle of
}\sigma}}f\left(  a_{j_{1}}a_{j_{2}}\cdots a_{j_{k}}\right)
\label{eq.darij2.p3.pf.0}%
\end{align}
and
\begin{align}
&  f_{n}\left(  a_{\tau\left(  1\right)  },a_{\tau\left(  2\right)  }%
,\ldots,a_{\tau\left(  n\right)  }\right) \nonumber\\
&  =\sum_{\sigma\in S_{n}}\left(  -1\right)  ^{\sigma}\prod
_{\substack{c=\left(  i_{1},i_{2},\ldots,i_{k}\right)  \\\text{is a cycle of
}\sigma}}f\left(  a_{\tau\left(  i_{1}\right)  }a_{\tau\left(  i_{2}\right)
}\cdots a_{\tau\left(  i_{k}\right)  }\right)  . \label{eq.darij2.p3.pf.1}%
\end{align}
However, recall that conjugate permutations have the same cycle type. More
concretely: If $\sigma\in S_{n}$ is any permutation, then the cycles of the
permutation $\tau\circ\sigma\circ\tau^{-1}$ are in bijection with the cycles
of $\sigma$: Namely, if $\left(  i_{1},i_{2},\ldots,i_{k}\right)  $ is a cycle
of $\sigma$, then $\left(  \tau\left(  i_{1}\right)  ,\tau\left(
i_{2}\right)  ,\ldots,\tau\left(  i_{k}\right)  \right)  $ is a cycle of
$\tau\circ\sigma\circ\tau^{-1}$, and this gives a 1-to-1 correspondence
between the cycles of $\sigma$ and the cycles of $\tau\circ\sigma\circ
\tau^{-1}$. Thus, if $\sigma\in S_{n}$ is any permutation, then
\begin{align*}
&  \prod_{\substack{c=\left(  i_{1},i_{2},\ldots,i_{k}\right)  \\\text{is a
cycle of }\sigma}}f\left(  a_{\tau\left(  i_{1}\right)  }a_{\tau\left(
i_{2}\right)  }\cdots a_{\tau\left(  i_{k}\right)  }\right) \\
&  =\prod_{\substack{c=\left(  j_{1},j_{2},\ldots,j_{k}\right)  \\\text{is a
cycle of }\tau\circ\sigma\circ\tau^{-1}}}f\left(  a_{j_{1}}a_{j_{2}}\cdots
a_{j_{k}}\right)  .
\end{align*}
Substituting this into \eqref{eq.darij2.p3.pf.1}, we obtain
\begin{align*}
&  f_{n}\left(  a_{\tau\left(  1\right)  },a_{\tau\left(  2\right)  }%
,\ldots,a_{\tau\left(  n\right)  }\right) \\
&  =\sum_{\sigma\in S_{n}}\underbrace{\left(  -1\right)  ^{\sigma}%
}_{\substack{=\left(  -1\right)  ^{\tau\circ\sigma\circ\tau^{-1}%
}\\\text{(since conjugate permutations}\\\text{have the same sign)}}%
}\ \ \prod_{\substack{c=\left(  j_{1},j_{2},\ldots,j_{k}\right)  \\\text{is a
cycle of }\tau\circ\sigma\circ\tau^{-1}}}f\left(  a_{j_{1}}a_{j_{2}}\cdots
a_{j_{k}}\right) \\
&  =\sum_{\sigma\in S_{n}}\left(  -1\right)  ^{\tau\circ\sigma\circ\tau^{-1}%
}\prod_{\substack{c=\left(  j_{1},j_{2},\ldots,j_{k}\right)  \\\text{is a
cycle of }\tau\circ\sigma\circ\tau^{-1}}}f\left(  a_{j_{1}}a_{j_{2}}\cdots
a_{j_{k}}\right) \\
&  =\sum_{\sigma\in S_{n}}\left(  -1\right)  ^{\sigma}\prod
_{\substack{c=\left(  j_{1},j_{2},\ldots,j_{k}\right)  \\\text{is a cycle of
}\sigma}}f\left(  a_{j_{1}}a_{j_{2}}\cdots a_{j_{k}}\right)
\end{align*}
(here, we have substituted $\sigma$ for $\tau\circ\sigma\circ\tau^{-1}$ in the
sum, because conjugation by $\tau$ is a bijection from $S_{n}$ onto itself).
Comparing this with \eqref{eq.darij2.p3.pf.0}, we obtain $f_{n}\left(
a_{\tau\left(  1\right)  },a_{\tau\left(  2\right)  },\ldots,a_{\tau\left(
n\right)  }\right)  =f_{n}\left(  a_{1},a_{2},\ldots,a_{n}\right)  $. This
proves Proposition \ref{prop.3}.
\end{proof}

The proof of Proposition \ref{prop.3} can be somewhat generalized, essentially
replacing the permutation $\tau$ by a bijection between two finite sets of
integers. We record the result, since it will prove useful later on:

\begin{proposition}
\label{prop.3g}Let $A$ be a ring, and let $B$ be a commutative ring. Let
$f:A\rightarrow B$ be a central $\mathbb{Z}$-linear map.

Let $U=\left\{  u_{1},u_{2},\ldots,u_{n}\right\}  $ be a finite set of
integers (with $u_{1},u_{2},\ldots,u_{n}$ distinct). Let $S_{U}$ denote the
group of all permutations of $U$. Let $a_{u}$ be an element of $A$ for each
$u\in U$. Then,%
\[
f_{n}\left(  a_{u_{1}},a_{u_{2}},\ldots,a_{u_{n}}\right)  =\sum_{\alpha\in
S_{U}}\left(  -1\right)  ^{\alpha}\prod\limits_{\substack{c=\left(
i_{1},i_{2},\ldots,i_{k}\right)  \\\text{is a cycle of }\alpha}}f\left(
a_{i_{1}}a_{i_{2}}\cdots a_{i_{k}}\right)  .
\]

\end{proposition}

\begin{proof}
Let $\tau:\left[  n\right]  \rightarrow U$ be the map that sends each
$i\in\left[  n\right]  $ to $u_{i}$. This map $\tau$ is a bijection (indeed,
it is surjective because $U=\left\{  u_{1},u_{2},\ldots,u_{n}\right\}  $, and
it is injective since $u_{1},u_{2},\ldots,u_{n}$ are distinct). Hence, for
each permutation $\sigma\in S_{n}$, the composition $\tau\circ\sigma\circ
\tau^{-1}$ is a permutation of $U$. Thus, we obtain a map%
\begin{align}
S_{n}  &  \rightarrow S_{U},\nonumber\\
\sigma &  \mapsto\tau\circ\sigma\circ\tau^{-1}, \label{pf.prop.3g.bij}%
\end{align}
which is easily seen to be a group isomorphism. Intuitively speaking,
$\tau\circ\sigma\circ\tau^{-1}$ is what you obtain if you take the permutation
$\sigma$ of $\left[  n\right]  $ and rename each element $i\in\left[
n\right]  $ (say, in the two-line notation of $\sigma$, or in the cycle
decomposition of $\sigma$) as $\tau\left(  i\right)  $. In particular, the
cycles of the permutation $\tau\circ\sigma\circ\tau^{-1}$ (for a given
$\sigma\in S_{n}$) are in bijection with the cycles of $\sigma$: Namely, if
$\left(  i_{1},i_{2},\ldots,i_{k}\right)  $ is a cycle of $\sigma$, then
$\left(  \tau\left(  i_{1}\right)  ,\tau\left(  i_{2}\right)  ,\ldots
,\tau\left(  i_{k}\right)  \right)  $ is a cycle of $\tau\circ\sigma\circ
\tau^{-1}$, and this gives a 1-to-1 correspondence between the cycles of
$\sigma$ and the cycles of $\tau\circ\sigma\circ\tau^{-1}$. Thus, if
$\sigma\in S_{n}$ is any permutation, then
\begin{align}
&  \prod_{\substack{c=\left(  i_{1},i_{2},\ldots,i_{k}\right)  \\\text{is a
cycle of }\sigma}}f\left(  a_{\tau\left(  i_{1}\right)  }a_{\tau\left(
i_{2}\right)  }\cdots a_{\tau\left(  i_{k}\right)  }\right) \nonumber\\
&  =\prod_{\substack{c=\left(  j_{1},j_{2},\ldots,j_{k}\right)  \\\text{is a
cycle of }\tau\circ\sigma\circ\tau^{-1}}}f\left(  a_{j_{1}}a_{j_{2}}\cdots
a_{j_{k}}\right)  . \label{pf.prop.3g.3}%
\end{align}
Moreover, if $\sigma\in S_{n}$ is any permutation, then%
\begin{equation}
\left(  -1\right)  ^{\sigma}=\left(  -1\right)  ^{\tau\circ\sigma\circ
\tau^{-1}}. \label{pf.prop.3g.4}%
\end{equation}
(This is easiest to see by factoring $\sigma$ into a product of
transpositions; see \cite[Exercise 5.12]{detnotes}.)

But each $i\in\left[  n\right]  $ satisfies $a_{u_{i}}=a_{\tau\left(
i\right)  }$ (since the definition of $\tau$ yields $\tau\left(  i\right)
=u_{i}$, thus $a_{\tau\left(  i\right)  }=a_{u_{i}}$). Therefore,%
\begin{align*}
&  f_{n}\left(  a_{u_{1}},a_{u_{2}},\ldots,a_{u_{n}}\right) \\
&  =f_{n}\left(  a_{\tau\left(  1\right)  },a_{\tau\left(  2\right)  }%
,\ldots,a_{\tau\left(  n\right)  }\right) \\
&  =\sum_{\sigma\in S_{n}}\underbrace{\left(  -1\right)  ^{\sigma}%
}_{\substack{=\left(  -1\right)  ^{\tau\circ\sigma\circ\tau^{-1}}\\\text{(by
\eqref{pf.prop.3g.4})}}}\ \ \underbrace{\prod_{\substack{c=\left(  i_{1}%
,i_{2},\ldots,i_{k}\right)  \\\text{is a cycle of }\sigma}}f\left(
a_{\tau\left(  i_{1}\right)  }a_{\tau\left(  i_{2}\right)  }\cdots
a_{\tau\left(  i_{k}\right)  }\right)  }_{\substack{=\prod
_{\substack{c=\left(  j_{1},j_{2},\ldots,j_{k}\right)  \\\text{is a cycle of
}\tau\circ\sigma\circ\tau^{-1}}}f\left(  a_{j_{1}}a_{j_{2}}\cdots a_{j_{k}%
}\right)  \\\text{(by \eqref{pf.prop.3g.3})}}}\\
&  \qquad\qquad\left(  \text{by \eqref{eq.darij1.1}, applied to }%
a_{\tau\left(  i\right)  }\text{ instead of }a_{i}\right) \\
&  =\sum_{\sigma\in S_{n}}\left(  -1\right)  ^{\tau\circ\sigma\circ\tau^{-1}%
}\prod_{\substack{c=\left(  j_{1},j_{2},\ldots,j_{k}\right)  \\\text{is a
cycle of }\tau\circ\sigma\circ\tau^{-1}}}f\left(  a_{j_{1}}a_{j_{2}}\cdots
a_{j_{k}}\right) \\
&  =\sum_{\alpha\in S_{U}}\left(  -1\right)  ^{\alpha}\prod
\limits_{\substack{c=\left(  j_{1},j_{2},\ldots,j_{k}\right)  \\\text{is a
cycle of }\alpha}}f\left(  a_{j_{1}}a_{j_{2}}\cdots a_{j_{k}}\right)
\end{align*}
(here, we have substituted $\alpha$ for $\tau\circ\sigma\circ\tau^{-1}$ in the
sum, since the map \eqref{pf.prop.3g.bij} is a bijection). Renaming the index
$\left(  j_{1},j_{2},\ldots,j_{k}\right)  $ as $\left(  i_{1},i_{2}%
,\ldots,i_{k}\right)  $ on the right hand side, we can rewrite this as%
\[
f_{n}\left(  a_{u_{1}},a_{u_{2}},\ldots,a_{u_{n}}\right)  =\sum_{\alpha\in
S_{U}}\left(  -1\right)  ^{\alpha}\prod\limits_{\substack{c=\left(
i_{1},i_{2},\ldots,i_{k}\right)  \\\text{is a cycle of }\alpha}}f\left(
a_{i_{1}}a_{i_{2}}\cdots a_{i_{k}}\right)  .
\]
Thus, Proposition \ref{prop.3g} is proved.
\end{proof}

\subsection{The sum of two $n$-homomorphisms}

The proof of Theorem \ref{thm.1} will require two lemmas:

\begin{lemma}
\label{lem.5}Let $A$ be a ring, and let $B$ be a commutative ring. Let
$f:A\rightarrow B$ be any central $\mathbb{Z}$-linear map. Let $m\in
\mathbb{N}$ be such that $f_{m}=0$. Then, $f_{n}=0$ for all $n\geq m$.
\end{lemma}

\begin{proof}
[Proof sketch.]If some positive integer $n$ satisfies $f_{n-1}=0$, then
$f_{n}=0$ as well (by the recursion \eqref{eq.darij1.2}). Hence, Lemma
\ref{lem.5} follows easily by induction on $n$.
\end{proof}

\begin{lemma}
\label{lem.0}Let $A$ and $B$ be two rings, with $B$ commutative. Let
$f:A\rightarrow B$ and $g:A\rightarrow B$ be two central $\mathbb{Z}$-linear
maps. Let $a_{1},a_{2},\ldots,a_{p}\in A$ be some elements.

For any subset $I=\left\{  i_{1},i_{2},\ldots,i_{k}\right\}  $ of $\left[
p\right]  $ (with $i_{1},i_{2},\ldots,i_{k}$ distinct) and any central map
$h:A\rightarrow B$, let us define
\[
h_{I}\left(  a\right)  :=h_{k}\left(  a_{i_{1}},a_{i_{2}},\ldots,a_{i_{k}%
}\right)  \in B.
\]
(This is well-defined, i.e., does not depend on the order in which we label
the elements of $I$ as $i_{1},i_{2},\ldots,i_{k}$, because Proposition
\ref{prop.3} shows that the map $h_{k}$ is symmetric.)

Then,%
\[
\left(  f+g\right)  _{p}\left(  a_{1},a_{2},\ldots,a_{p}\right)
=\sum\limits_{U\sqcup V=\left[  p\right]  }f_{U}\left(  a\right)  \cdot
g_{V}\left(  a\right)  .
\]
Here, the notation \textquotedblleft$U\sqcup V=\left[  p\right]
$\textquotedblright\ means that $U$ and $V$ are two disjoint sets whose union
is $\left[  p\right]  $ (that is, $U$ is a subset of $\left[  p\right]  $, and
$V$ is its complement in $\left[  p\right]  $).
\end{lemma}

\begin{proof}
[Proof of Lemma \ref{lem.0} (sketched).]From \eqref{eq.darij1.1} (applied to
$f+g$ and $p$ instead of $f$ and $n$), we have%
\begin{align}
&  \left(  f+g\right)  _{p}\left(  a_{1},a_{2},\ldots,a_{p}\right) \nonumber\\
&  =\sum_{\sigma\in S_{p}}\left(  -1\right)  ^{\sigma}\prod
_{\substack{c=\left(  i_{1},i_{2},\ldots,i_{k}\right)  \\\text{is a cycle of
}\sigma}}\ \ \underbrace{\left(  f+g\right)  \left(  a_{i_{1}}a_{i_{2}}\cdots
a_{i_{k}}\right)  }_{=f\left(  a_{i_{1}}a_{i_{2}}\cdots a_{i_{k}}\right)
+g\left(  a_{i_{1}}a_{i_{2}}\cdots a_{i_{k}}\right)  }\nonumber\\
&  =\sum_{\sigma\in S_{p}}\left(  -1\right)  ^{\sigma}\prod
_{\substack{c=\left(  i_{1},i_{2},\ldots,i_{k}\right)  \\\text{is a cycle of
}\sigma}}\left(  f\left(  a_{i_{1}}a_{i_{2}}\cdots a_{i_{k}}\right)  +g\left(
a_{i_{1}}a_{i_{2}}\cdots a_{i_{k}}\right)  \right)
.\ \ \ \ \ \ \ \ \ \ \label{pf.lem.0.1}%
\end{align}

Now, fix a permutation $\sigma\in S_{p}$. If we expand the product%
\begin{equation}
\prod_{\substack{c=\left(  i_{1},i_{2},\ldots,i_{k}\right)  \\\text{is a cycle
of }\sigma}}\left(  f\left(  a_{i_{1}}a_{i_{2}}\cdots a_{i_{k}}\right)
+g\left(  a_{i_{1}}a_{i_{2}}\cdots a_{i_{k}}\right)  \right)  ,
\label{pf.lem.0.prod}%
\end{equation}
then we obtain a sum over all ways to choose, for each cycle $c=\left(
i_{1},i_{2},\ldots,i_{k}\right)  $ of $\sigma$, one of the two addends of the
sum $f\left(  a_{i_{1}}a_{i_{2}}\cdots a_{i_{k}}\right)  +g\left(  a_{i_{1}%
}a_{i_{2}}\cdots a_{i_{k}}\right)  $. Such \textquotedblleft ways to
choose\textquotedblright\ can be viewed as colorings of the cycles of $\sigma
$, in which each cycle is either colored red (meaning that the $f\left(
a_{i_{1}}a_{i_{2}}\cdots a_{i_{k}}\right)  $ addend is chosen) or colored blue
(meaning that the $g\left(  a_{i_{1}}a_{i_{2}}\cdots a_{i_{k}}\right)  $
addend is chosen). Thus, the expanded form of \eqref{pf.lem.0.prod} can be
written as follows:%
\begin{align*}
&  \prod_{\substack{c=\left(  i_{1},i_{2},\ldots,i_{k}\right)  \\\text{is a
cycle of }\sigma}}\left(  f\left(  a_{i_{1}}a_{i_{2}}\cdots a_{i_{k}}\right)
+g\left(  a_{i_{1}}a_{i_{2}}\cdots a_{i_{k}}\right)  \right) \\
&  =\sum_{\substack{\text{coloring of all cycles of }\sigma\\\text{in red and
blue}}}\ \ \prod_{\substack{c=\left(  i_{1},i_{2},\ldots,i_{k}\right)
\\\text{is a cycle of }\sigma}}%
\begin{cases}
f\left(  a_{i_{1}}a_{i_{2}}\cdots a_{i_{k}}\right)  , & \text{if }c\text{ is
red};\\
g\left(  a_{i_{1}}a_{i_{2}}\cdots a_{i_{k}}\right)  , & \text{if }c\text{ is
blue.}%
\end{cases}
\end{align*}

Equivalently, instead of coloring cycles, we can just as well color the
\textbf{elements} of these cycles red and blue (viz., all elements of all red
cycles are colored red, while all elements of all blue cycles are colored
blue). These colorings are not arbitrary, but must have the property that each
cycle is either completely red (i.e., all its elements are red) or completely
blue (i.e., all its elements are blue); in other words, they must have the
property that $\sigma$ sends red elements to red elements and blue elements to
blue elements. Thus, the above expanded form of \eqref{pf.lem.0.prod} can be
rewritten as follows:%
\begin{align*}
&  \prod_{\substack{c=\left(  i_{1},i_{2},\ldots,i_{k}\right)  \\\text{is a
cycle of }\sigma}}\left(  f\left(  a_{i_{1}}a_{i_{2}}\cdots a_{i_{k}}\right)
+g\left(  a_{i_{1}}a_{i_{2}}\cdots a_{i_{k}}\right)  \right) \\
&  =\sum_{\substack{\text{coloring of all elements of }\left[  p\right]
\\\text{in red and blue;}\\\sigma\text{ sends red elements to red
elements}\\\text{and blue elements to blue elements}}}\ \ \prod
_{\substack{c=\left(  i_{1},i_{2},\ldots,i_{k}\right)  \\\text{is a cycle of
}\sigma}}%
\begin{cases}
f\left(  a_{i_{1}}a_{i_{2}}\cdots a_{i_{k}}\right)  , & \text{if }c\text{ is
red};\\
g\left(  a_{i_{1}}a_{i_{2}}\cdots a_{i_{k}}\right)  , & \text{if }c\text{ is
blue}%
\end{cases}
\end{align*}
(where \textquotedblleft$c$ is red\textquotedblright\ means that all elements
of $c$ are red, and likewise for \textquotedblleft blue\textquotedblright).

Of course, a coloring of all elements of $\left[  p\right]  $ in red and blue
is the same thing as a decomposition of $\left[  p\right]  $ into two disjoint
subsets $U$ and $V$ (where $U$ is the set of all red elements and $V$ is the
set of all blue elements); in other words, it is a choice of two sets $U$ and
$V$ such that $U\sqcup V=\left[  p\right]  $. Moreover, the condition
\textquotedblleft$\sigma$ sends red elements to red elements\textquotedblright%
\ is simply saying that $\sigma\left(  U\right)  \subseteq U$, whereas the
condition \textquotedblleft$\sigma$ sends blue elements to blue
elements\textquotedblright\ is saying that $\sigma\left(  V\right)  \subseteq
V$. Hence, our above expanded form of \eqref{pf.lem.0.prod} can be rewritten
as follows:%
\begin{align}
&  \prod_{\substack{c=\left(  i_{1},i_{2},\ldots,i_{k}\right)  \\\text{is a
cycle of }\sigma}}\left(  f\left(  a_{i_{1}}a_{i_{2}}\cdots a_{i_{k}}\right)
+g\left(  a_{i_{1}}a_{i_{2}}\cdots a_{i_{k}}\right)  \right) \nonumber\\
&  =\sum_{\substack{U\sqcup V=\left[  p\right]  ;\\\sigma\left(  U\right)
\subseteq U;\\\sigma\left(  V\right)  \subseteq V}}\ \ \prod
_{\substack{c=\left(  i_{1},i_{2},\ldots,i_{k}\right)  \\\text{is a cycle of
}\sigma}}%
\begin{cases}
f\left(  a_{i_{1}}a_{i_{2}}\cdots a_{i_{k}}\right)  , & \text{if }c\subseteq
U;\\
g\left(  a_{i_{1}}a_{i_{2}}\cdots a_{i_{k}}\right)  , & \text{if }c\subseteq V
\end{cases}
\ \ \ \ \ \ \ \ \ \ \label{pf.lem.0.prod4}%
\end{align}
(where \textquotedblleft$c\subseteq U$\textquotedblright\ means that all
elements of $c$ belong to $U$, and likewise for \textquotedblleft$c\subseteq
V$\textquotedblright).

Forget that we fixed $\sigma$. We thus have proved \eqref{pf.lem.0.prod4} for
each $\sigma\in S_{p}$. Now, \eqref{pf.lem.0.1} becomes%
\begin{align}
&  \left(  f+g\right)  _{p}\left(  a_{1},a_{2},\ldots,a_{p}\right) \nonumber\\
&  =\sum_{\sigma\in S_{p}}\left(  -1\right)  ^{\sigma}\prod
_{\substack{c=\left(  i_{1},i_{2},\ldots,i_{k}\right)  \\\text{is a cycle of
}\sigma}}\left(  f\left(  a_{i_{1}}a_{i_{2}}\cdots a_{i_{k}}\right)  +g\left(
a_{i_{1}}a_{i_{2}}\cdots a_{i_{k}}\right)  \right) \nonumber\\
&  =\sum_{\sigma\in S_{p}}\left(  -1\right)  ^{\sigma}\sum_{\substack{U\sqcup
V=\left[  p\right]  ;\\\sigma\left(  U\right)  \subseteq U;\\\sigma\left(
V\right)  \subseteq V}}\ \ \prod_{\substack{c=\left(  i_{1},i_{2},\ldots
,i_{k}\right)  \\\text{is a cycle of }\sigma}}%
\begin{cases}
f\left(  a_{i_{1}}a_{i_{2}}\cdots a_{i_{k}}\right)  , & \text{if }c\subseteq
U;\\
g\left(  a_{i_{1}}a_{i_{2}}\cdots a_{i_{k}}\right)  , & \text{if }c\subseteq V
\end{cases}
\nonumber\\
&  \qquad\qquad\left(  \text{by \eqref{pf.lem.0.prod4}}\right) \nonumber\\
&  =\sum_{U\sqcup V=\left[  p\right]  }\ \ \sum_{\substack{\sigma\in
S_{p};\\\sigma\left(  U\right)  \subseteq U;\\\sigma\left(  V\right)
\subseteq V}}\left(  -1\right)  ^{\sigma}\prod_{\substack{c=\left(
i_{1},i_{2},\ldots,i_{k}\right)  \\\text{is a cycle of }\sigma}}%
\begin{cases}
f\left(  a_{i_{1}}a_{i_{2}}\cdots a_{i_{k}}\right)  , & \text{if }c\subseteq
U;\\
g\left(  a_{i_{1}}a_{i_{2}}\cdots a_{i_{k}}\right)  , & \text{if }c\subseteq V
\end{cases}
\ \ \ \ \ \ \ \ \ \ \label{pf.lem.0.3}%
\end{align}
(here, we have interchanged the two summation signs).

Now, fix a decomposition $U\sqcup V=\left[  p\right]  $ of the set $\left[
p\right]  $ into two disjoint subsets $U$ and $V$. Then, any permutation
$\sigma\in S_{p}$ satisfying $\sigma\left(  U\right)  \subseteq U$ and
$\sigma\left(  V\right)  \subseteq V$ can be \textquotedblleft
decomposed\textquotedblright\ into a pair $\left(  \alpha,\beta\right)  $
consisting of a permutation $\alpha:=\sigma\mid_{U}$ of $U$ and a permutation
$\beta:=\sigma\mid_{V}$ of $V$. To put it more formally: There is a bijection%
\begin{align*}
&  \text{from }\left\{  \text{permutations }\sigma\in S_{p}\text{ satisfying
}\sigma\left(  U\right)  \subseteq U\text{ and }\sigma\left(  V\right)
\subseteq V\right\} \\
&  \text{to }S_{U}\times S_{V}%
\end{align*}
(where $S_{X}$ denotes the group of permutations of any set $X$) that sends
each $\sigma$ to $\left(  \alpha,\beta\right)  :=\left(  \sigma\mid
_{U},\ \sigma\mid_{V}\right)  $. Moreover, the cycles of the former
permutation $\sigma$ are simply the cycles of the two \textquotedblleft
component\textquotedblright\ permutations $\alpha$ and $\beta$, and it is easy
to see that the sign of $\sigma$ is given by%
\[
\left(  -1\right)  ^{\sigma}=\left(  -1\right)  ^{\alpha}\left(  -1\right)
^{\beta}%
\]
(this follows, e.g., by writing $\alpha$ and $\beta$ as products of
transpositions). Hence,
\begin{align}
&  \sum_{\substack{\sigma\in S_{p};\\\sigma\left(  U\right)  \subseteq
U;\\\sigma\left(  V\right)  \subseteq V}}\left(  -1\right)  ^{\sigma}%
\prod_{\substack{c=\left(  i_{1},i_{2},\ldots,i_{k}\right)  \\\text{is a cycle
of }\sigma}}%
\begin{cases}
f\left(  a_{i_{1}}a_{i_{2}}\cdots a_{i_{k}}\right)  , & \text{if }c\subseteq
U;\\
g\left(  a_{i_{1}}a_{i_{2}}\cdots a_{i_{k}}\right)  , & \text{if }c\subseteq V
\end{cases}
\nonumber\\
&  =\sum_{\left(  \alpha,\beta\right)  \in S_{U}\times S_{V}}\left(
-1\right)  ^{\alpha}\left(  -1\right)  ^{\beta}\prod_{\substack{c=\left(
i_{1},i_{2},\ldots,i_{k}\right)  \\\text{is a cycle of }\alpha\text{ or of
}\beta}}%
\begin{cases}
f\left(  a_{i_{1}}a_{i_{2}}\cdots a_{i_{k}}\right)  , & \text{if }c\subseteq
U;\\
g\left(  a_{i_{1}}a_{i_{2}}\cdots a_{i_{k}}\right)  , & \text{if }c\subseteq V
\end{cases}
\nonumber\\
&  =\sum_{\left(  \alpha,\beta\right)  \in S_{U}\times S_{V}}\left(
-1\right)  ^{\alpha}\left(  -1\right)  ^{\beta}\left(  \prod
_{\substack{c=\left(  i_{1},i_{2},\ldots,i_{k}\right)  \\\text{is a cycle of
}\alpha}}f\left(  a_{i_{1}}a_{i_{2}}\cdots a_{i_{k}}\right)  \right)
\nonumber\\
&  \qquad\left(  \prod_{\substack{c=\left(  i_{1},i_{2},\ldots,i_{k}\right)
\\\text{is a cycle of }\beta}}g\left(  a_{i_{1}}a_{i_{2}}\cdots a_{i_{k}%
}\right)  \right) \nonumber\\
&  =\left(  \sum_{\alpha\in S_{U}}\left(  -1\right)  ^{\alpha}\prod
\limits_{\substack{c=\left(  i_{1},i_{2},\ldots,i_{k}\right)  \\\text{is a
cycle of }\alpha}}f\left(  a_{i_{1}}a_{i_{2}}\cdots a_{i_{k}}\right)  \right)
\nonumber\\
&  \qquad\left(  \sum_{\beta\in S_{V}}\left(  -1\right)  ^{\beta}%
\prod\limits_{\substack{c=\left(  i_{1},i_{2},\ldots,i_{k}\right)  \\\text{is
a cycle of }\beta}}g\left(  a_{i_{1}}a_{i_{2}}\cdots a_{i_{k}}\right)
\right)  . \label{pf.lem.0.5}%
\end{align}

However, if we write the set $U$ in the form $U=\left\{  u_{1},u_{2}%
,\ldots,u_{r}\right\}  $ (with $u_{1},u_{2},\ldots,u_{r}$ distinct), then%
\begin{align}
f_{U}\left(  a\right)   &  =f_{r}\left(  a_{u_{1}},a_{u_{2}},\ldots,a_{u_{r}%
}\right)  \qquad\left(  \text{by the definition of }f_{U}\right) \nonumber\\
&  =\sum_{\alpha\in S_{U}}\left(  -1\right)  ^{\alpha}\prod
\limits_{\substack{c=\left(  i_{1},i_{2},\ldots,i_{k}\right)  \\\text{is a
cycle of }\alpha}}f\left(  a_{i_{1}}a_{i_{2}}\cdots a_{i_{k}}\right)
\label{pf.lem.0.7}%
\end{align}
(by Proposition \ref{prop.3g}, applied to $r$ instead of $n$). Thus, we have
shown that%
\[
\sum_{\alpha\in S_{U}}\left(  -1\right)  ^{\alpha}\prod
\limits_{\substack{c=\left(  i_{1},i_{2},\ldots,i_{k}\right)  \\\text{is a
cycle of }\alpha}}f\left(  a_{i_{1}}a_{i_{2}}\cdots a_{i_{k}}\right)
=f_{U}\left(  a\right)  .
\]
Similarly,%
\[
\sum_{\beta\in S_{V}}\left(  -1\right)  ^{\beta}\prod
\limits_{\substack{c=\left(  i_{1},i_{2},\ldots,i_{k}\right)  \\\text{is a
cycle of }\beta}}g\left(  a_{i_{1}}a_{i_{2}}\cdots a_{i_{k}}\right)
=g_{V}\left(  a\right)  .
\]

Substituting these two equalities into \eqref{pf.lem.0.5}, we obtain%
\begin{align}
&  \sum_{\substack{\sigma\in S_{p};\\\sigma\left(  U\right)  \subseteq
U;\\\sigma\left(  V\right)  \subseteq V}}\left(  -1\right)  ^{\sigma}%
\prod_{\substack{c=\left(  i_{1},i_{2},\ldots,i_{k}\right)  \\\text{is a cycle
of }\sigma}}%
\begin{cases}
f\left(  a_{i_{1}}a_{i_{2}}\cdots a_{i_{k}}\right)  , & \text{if }c\subseteq
U;\\
g\left(  a_{i_{1}}a_{i_{2}}\cdots a_{i_{k}}\right)  , & \text{if }c\subseteq V
\end{cases}
\nonumber\\
&  =f_{U}\left(  a\right)  \cdot g_{V}\left(  a\right)  . \label{pf.lem.0.9}%
\end{align}

Now forget that we fixed the decomposition $U\sqcup V=\left[  p\right]  $.
Now, substituting \eqref{pf.lem.0.9} into \eqref{pf.lem.0.3}, we obtain%
\[
\left(  f+g\right)  _{p}\left(  a_{1},a_{2},\ldots,a_{p}\right)
=\sum\limits_{U\sqcup V=\left[  p\right]  }f_{U}\left(  a\right)  \cdot
g_{V}\left(  a\right)  .
\]
This proves Lemma \ref{lem.0}.
\end{proof}

\begin{proof}
[Proof of Theorem \ref{thm.1} (sketched).]The map $f$ is central (being an
$n$-homomorphism). Similarly, $g$ is central. Thus, it is easy to see that
$f+g$ is central. It remains to show that $\left(  f+g\right)  _{n+m+1}=0$.

Let $p:=n+m+1$. Let $a_{1},a_{2},\ldots,a_{p}\in A$. Then, Lemma \ref{lem.0}
yields%
\begin{equation}
\left(  f+g\right)  _{p}\left(  a_{1},a_{2},\ldots,a_{p}\right)
=\sum\limits_{U\sqcup V=\left[  p\right]  }f_{U}\left(  a\right)  \cdot
g_{V}\left(  a\right)  , \label{pf.thm.1.1}%
\end{equation}
where the notations are as defined in Lemma \ref{lem.0}.

Now, fix any decomposition $U\sqcup V=\left[  p\right]  $ of $\left[
p\right]  $ into two disjoint subsets $U$ and $V$. We shall show that%
\begin{equation}
f_{U}\left(  a\right)  \cdot g_{V}\left(  a\right)  =0. \label{pf.thm.1.2}%
\end{equation}
[\textit{Proof:} From $U\sqcup V=\left[  p\right]  $, we obtain%
\[
\left\vert U\right\vert +\left\vert V\right\vert =\left\vert U\sqcup
V\right\vert =\left\vert \left[  p\right]  \right\vert =p=n+m+1>n+m.
\]
Hence, we must have $\left\vert U\right\vert >n$ or $\left\vert V\right\vert
>m$ (since otherwise, we would have both $\left\vert U\right\vert \leq n$ and
$\left\vert V\right\vert \leq m$, and therefore we could add these two
inequalities together and obtain $\left\vert U\right\vert +\left\vert
V\right\vert \leq n+m$, which would contradict $\left\vert U\right\vert
+\left\vert V\right\vert >n+m$). We WLOG assume that $\left\vert U\right\vert
>n$ (since the $\left\vert V\right\vert >m$ case is analogous). Thus,
$\left\vert U\right\vert \geq n+1$. But $f_{n+1}=0$ (since $f$ is an
$n$-homomorphism). Hence, Lemma \ref{lem.5} (applied to $\left\vert
U\right\vert $ and $n+1$ instead of $n$ and $m$) yields $f_{\left\vert
U\right\vert }=0$ (since $\left\vert U\right\vert \geq n+1$). Now, let us
write the set $U$ in the form $U=\left\{  u_{1},u_{2},\ldots,u_{r}\right\}  $
(with $u_{1},u_{2},\ldots,u_{r}$ distinct). Then, $r=\left\vert U\right\vert $
and therefore $f_{r}=f_{\left\vert U\right\vert }=0$. But the definition of
$f_{U}$ (see Lemma \ref{lem.0}) yields%
\[
f_{U}\left(  a\right)  =\underbrace{f_{r}}_{=0}\left(  a_{u_{1}},a_{u_{2}%
},\ldots,a_{u_{r}}\right)  =0.
\]
Hence, $f_{U}\left(  a\right)  \cdot g_{V}\left(  a\right)  =0\cdot
g_{V}\left(  a\right)  =0$. This proves \eqref{pf.thm.1.2}.] \medskip

Forget that we fixed the decomposition $U\sqcup V=\left[  p\right]  $. We have
thus shown that \eqref{pf.thm.1.2} holds for each such decomposition. In other
words, all addends on the right hand side of \eqref{pf.thm.1.1} are $0$.
Hence, the whole right hand side is $0$, and so we can rewrite
\eqref{pf.thm.1.1} as%
\[
\left(  f+g\right)  _{p}\left(  a_{1},a_{2},\ldots,a_{p}\right)  =0.
\]
Since $a_{1},a_{2},\ldots,a_{p}$ were chosen arbitrarily, this proves that
$\left(  f+g\right)  _{p}=0$. In other words, $\left(  f+g\right)  _{n+m+1}=0$
(since $p=n+m+1$). This completes the proof of Theorem \ref{thm.1}.
\end{proof}

\subsection{The composition formula}

Next, we introduce some notations. A \emph{set partition} (henceforth just
\emph{partition}) of a finite set $J$ means a set $\left\{  J_{1},J_{2}%
,\ldots,J_{k}\right\}  $ of disjoint nonempty subsets of $J$ whose union is
$J_{1}\cup J_{2}\cup\cdots\cup J_{k}=J$. These subsets $J_{1},J_{2}%
,\ldots,J_{k}$ are called the \emph{blocks} of the partition. For instance,
the five partitions of $\left[  3\right]  $ are\footnote{Recall that for any
integer $r\geq0$, we let $\left[  r\right]  $ denote the set $\left\{
1,2,\ldots,r\right\}  $.}
\begin{align*}
&  \left\{  \left\{  1,2,3\right\}  \right\}  \qquad\left(  \text{with
}1\text{ block}\right)  \qquad\text{and}\\
&  \left\{  \left\{  1,2\right\}  ,\left\{  3\right\}  \right\}  \text{ and
}\left\{  \left\{  1,3\right\}  ,\left\{  2\right\}  \right\}  \text{ and
}\left\{  \left\{  2,3\right\}  ,\left\{  1\right\}  \right\} \\
&  \qquad\qquad\qquad\text{(with }2\text{ blocks each)}\qquad\text{and}\\
&  \left\{  \left\{  1\right\}  ,\left\{  2\right\}  ,\left\{  3\right\}
\right\}  \qquad\left(  \text{with }3\text{ blocks}\right)  .
\end{align*}
Note that $\left\{  \left\{  1,3\right\}  ,\left\{  2\right\}  \right\}  $ and
$\left\{  \left\{  2\right\}  ,\left\{  3,1\right\}  \right\}  $ are the same
partition. We let $\Pi_{J}$ denote the set of all partitions of the finite set
$J$. Note that the empty set $\varnothing$ has a unique partition, which has
$0$ blocks; that is, $\Pi_{\varnothing}=\left\{  \varnothing\right\}  $.

Let us agree to always write partitions with their blocks distinct: e.g., when
we say \textquotedblleft$\left\{  P_{1},P_{2},\ldots,P_{k}\right\}  \in\Pi
_{J}$\textquotedblright, we shall understand that $P_{1},P_{2},\ldots,P_{k}$
are distinct.

The following formula (discovered by GPT-5.4) will be crucial to our proof of
Theorem \ref{thm.6}:

\begin{theorem}
\label{thm.4}Let $A$, $B$ and $C$ be three rings, with $B$ and $C$
commutative. Let $g:A\rightarrow B$ and $f:B\rightarrow C$ be two central
$\mathbb{Z}$-linear maps. Let $h:=f\circ g:A\rightarrow C$. Let $a_{1}%
,a_{2},\ldots,a_{n}\in A$. For any subset $P=\left\{  p_{1},p_{2},\ldots
,p_{r}\right\}  $ of $\left[  n\right]  $ (with $p_{1},p_{2},\ldots,p_{r}$
distinct), we set
\begin{equation}
b_{P}:=g_{r}\left(  a_{p_{1}},a_{p_{2}},\ldots,a_{p_{r}}\right)  .
\label{eq.darij2.t4.bP}%
\end{equation}
(This is well-defined, i.e., does not depend on the order in which we label
the elements of $P$ as $p_{1},p_{2},\ldots,p_{r}$, because Proposition
\ref{prop.3} shows that the map $g_{r}$ is symmetric.) Then,
\[
h_{n}\left(  a_{1},a_{2},\ldots,a_{n}\right)  =\sum_{\pi=\left\{  P_{1}%
,P_{2},\ldots,P_{k}\right\}  \in\Pi_{\left[  n\right]  }}f_{k}\left(
b_{P_{1}},b_{P_{2}},\ldots,b_{P_{k}}\right)  .
\]
(The expression $f_{k}\left(  b_{P_{1}},b_{P_{2}},\ldots,b_{P_{k}}\right)  $
here is well-defined, i.e., does not depend on the order in which we label the
blocks of $\pi$ as $P_{1},P_{2},\ldots,P_{k}$, because Proposition
\ref{prop.3} shows that the map $f_{k}$ is symmetric.)
\end{theorem}

\begin{example}
Let us set $n=3$ in Theorem \ref{thm.4}. Then, the theorem says that%
\begin{align*}
&  h_{3}\left(  a_{1},a_{2},a_{3}\right) \\
&  =\sum_{\pi=\left\{  P_{1},P_{2},\ldots,P_{k}\right\}  \in\Pi_{\left[
3\right]  }}f_{k}\left(  b_{P_{1}},b_{P_{2}},\ldots,b_{P_{k}}\right) \\
&  =f_{1}\left(  b_{\left\{  1,2,3\right\}  }\right)  +f_{2}\left(
b_{\left\{  1,2\right\}  },b_{\left\{  3\right\}  }\right)  +f_{2}\left(
b_{\left\{  1,3\right\}  },b_{\left\{  2\right\}  }\right) \\
&  \ \ \ \ \ \ \ \ \ \ +f_{2}\left(  b_{\left\{  2,3\right\}  },b_{\left\{
1\right\}  }\right)  +f_{3}\left(  b_{\left\{  1\right\}  },b_{\left\{
2\right\}  },b_{\left\{  3\right\}  }\right) \\
&  =f_{1}\left(  g_{3}\left(  a_{1},a_{2},a_{3}\right)  \right)  +f_{2}\left(
g_{2}\left(  a_{1},a_{2}\right)  ,g_{1}\left(  a_{3}\right)  \right)
+f_{2}\left(  g_{2}\left(  a_{1},a_{3}\right)  ,g_{1}\left(  a_{2}\right)
\right) \\
&  \ \ \ \ \ \ \ \ \ \ +f_{2}\left(  g_{2}\left(  a_{2},a_{3}\right)
,g_{1}\left(  a_{1}\right)  \right)  +f_{3}\left(  g_{1}\left(  a_{1}\right)
,g_{1}\left(  a_{2}\right)  ,g_{1}\left(  a_{3}\right)  \right)  .
\end{align*}

\end{example}

\begin{proof}
[Proof of Theorem \ref{thm.4}.]I would welcome an inclusion/exclusion argument
using cycles of permutations, but GPT-5.4 suggests an induction proof instead,
and that shall do.

We induct on $n$.

\textit{Base case:} For $n=0$, the claim of Theorem \ref{thm.4} is saying that
$h_{0}\left(  {}\right)  =f_{0}\left(  {}\right)  $ (since $\Pi_{\left[
0\right]  }=\Pi_{\varnothing}=\left\{  \varnothing\right\}  $), which is clear
because both sides equal $1$.

Strictly speaking, this is enough to complete the base case, but let us also
check the $n=1$ case.

For $n=1$, the claim of Theorem \ref{thm.4} is saying that $h_{1}\left(
a_{1}\right)  =f_{1}\left(  b_{\left\{  1\right\}  }\right)  $. Since
$f_{1}=f$ and $h_{1}=h$ and $b_{\left\{  1\right\}  }=g_{1}\left(
a_{1}\right)  =g\left(  a_{1}\right)  $, this boils down to $h\left(
a_{1}\right)  =f\left(  g\left(  a_{1}\right)  \right)  $, which follows from
$h=f\circ g$. Thus, the base case is proved.

\textit{Induction step:} Let $n\geq2$. Assume (as induction hypothesis) that
Theorem \ref{thm.4} holds for $n-1$ instead of $n$. We must now prove it for
$n$.

The induction hypothesis yields
\begin{equation}
h_{n-1}\left(  a_{1},a_{2},\ldots,a_{n-1}\right)  =\sum_{\pi=\left\{
P_{1},P_{2},\ldots,P_{k}\right\}  \in\Pi_{\left[  n-1\right]  }}f_{k}\left(
b_{P_{1}},b_{P_{2}},\ldots,b_{P_{k}}\right)
.\ \ \ \ \ \ \ \ \ \ \label{eq.darij2.t4.pf.IH1}%
\end{equation}
Moreover, for each $i\in\left[  n-1\right]  $, we define $n-1$ elements
$a_{1}^{(i)},a_{2}^{(i)},\ldots,a_{n-1}^{(i)}$ of $A$ by
\begin{align}
&  \left(  a_{1}^{(i)},a_{2}^{(i)},\ldots,a_{n-1}^{(i)}\right) \nonumber\\
&  :=\left(  a_{1},a_{2},\ldots,a_{i-1},a_{i}a_{n},a_{i+1},a_{i+2}%
,\ldots,a_{n-1}\right)  \label{eq.darij2.t4.pf.a1i}%
\end{align}
(these are what we called $a_{1}^{\prime},a_{2}^{\prime},\ldots,a_{n-1}%
^{\prime}$ in \eqref{eq.darij2.p2.pf.5a}), and then the induction hypothesis
(applied to these $n-1$ elements) yields
\begin{equation}
h_{n-1}\left(  a_{1}^{(i)},a_{2}^{(i)},\ldots,a_{n-1}^{(i)}\right)  =\sum
_{\pi=\left\{  P_{1},P_{2},\ldots,P_{k}\right\}  \in\Pi_{\left[  n-1\right]
}}f_{k}\left(  b_{P_{1}}^{(i)},b_{P_{2}}^{(i)},\ldots,b_{P_{k}}^{(i)}\right)
,\ \ \ \ \ \ \ \ \ \ \label{eq.darij2.t4.pf.IH2}%
\end{equation}
where for any subset $P=\left\{  p_{1},p_{2},\ldots,p_{r}\right\}  $ of
$\left[  n-1\right]  $ (with $p_{1},p_{2},\ldots,p_{r}$ distinct), we set
\begin{equation}
b_{P}^{(i)}:=g_{r}\left(  a_{p_{1}}^{(i)},a_{p_{2}}^{(i)},\ldots,a_{p_{r}%
}^{(i)}\right)  . \label{eq.darij2.t4.pf.b1i}%
\end{equation}

Note that if $P$ is a subset of $\left[  n-1\right]  $ and if $i\in\left[
n-1\right]  $ is such that $i\notin P$, then
\begin{equation}
b_{P}=b_{P}^{(i)}. \label{eq.darij2.t4.pf.bP=bPi}%
\end{equation}
(Indeed, let $P$ be a subset of $\left[  n-1\right]  $, and let $i\in\left[
n-1\right]  $ be such that $i\notin P$. Then, writing $P$ as $P=\left\{
p_{1},p_{2},\ldots,p_{r}\right\}  $ (with $p_{1},p_{2},\ldots,p_{r}$
distinct), we see that each $k\in\left[  r\right]  $ satisfies $p_{k}\neq i$
(since $i\notin P=\left\{  p_{1},p_{2},\ldots,p_{r}\right\}  $) and therefore
$a_{p_{k}}^{(i)}=a_{p_{k}}$ (since \eqref{eq.darij2.t4.pf.a1i} shows that
$a_{j}^{(i)}=a_{j}$ for all $j\neq i$). Hence, the right hand side of
\eqref{eq.darij2.t4.pf.b1i} equals the right hand side of
\eqref{eq.darij2.t4.bP}. Therefore, the same holds for the left hand sides as
well. In other words, we have $b_{P}^{(i)}=b_{P}$. This proves \eqref{eq.darij2.t4.pf.bP=bPi}.)

Note also that \eqref{eq.darij2.t4.bP} yields $b_{\left\{  n\right\}  }%
=g_{1}\left(  a_{n}\right)  =g\left(  a_{n}\right)  $.

Next, we observe that each subset $P$ of $\left[  n-1\right]  $ satisfies
\begin{equation}
b_{P\cup\left\{  n\right\}  }=b_{P}b_{\left\{  n\right\}  }-\sum_{i\in P}%
b_{P}^{(i)}. \label{eq.darij2.t4.pf.bPn}%
\end{equation}
(Indeed, let $P$ be a subset of $\left[  n-1\right]  $. Write $P$ as
$P=\left\{  p_{1},p_{2},\ldots,p_{r}\right\}  $ (with $p_{1},p_{2}%
,\ldots,p_{r}$ distinct). Then $P\cup\left\{  n\right\}  =\left\{  p_{1}%
,p_{2},\ldots,p_{r},n\right\}  $ (with $p_{1},p_{2},\ldots,p_{r},n$ distinct
because $P\subseteq\left[  n-1\right]  $), hence
\begin{equation}
b_{P\cup\left\{  n\right\}  }=g_{r+1}\left(  a_{p_{1}},a_{p_{2}}%
,\ldots,a_{p_{r}},a_{n}\right)  \qquad\left(  \text{by
\eqref{eq.darij2.t4.bP}}\right)  . \label{eq.darij2.t4.pf.bPn.pf.0}%
\end{equation}
But each $i\in\left[  r\right]  $ satisfies
\begin{align}
b_{P}^{(p_{i})}  &  =g_{r}\underbrace{\left(  a_{p_{1}}^{(p_{i})},a_{p_{2}%
}^{(p_{i})},\ldots,a_{p_{r}}^{(p_{i})}\right)  }_{\substack{=\left(  a_{p_{1}%
},a_{p_{2}},\ldots,a_{p_{i-1}},a_{p_{i}}a_{n},a_{p_{i+1}},a_{p_{i+2}}%
,\ldots,a_{p_{r}}\right)  \\\text{(indeed, all }k\neq i\text{ satisfy }%
p_{k}\neq p_{i}\text{ (since }p_{1},p_{2},\ldots,p_{r}\text{ are
distinct)}\\\text{and thus }a_{p_{k}}^{(p_{i})}=a_{p_{k}}\text{ (by
\eqref{eq.darij2.t4.pf.a1i}),}\\\text{whereas }a_{p_{i}}^{(p_{i})}=a_{p_{i}%
}a_{n}\text{ (again by \eqref{eq.darij2.t4.pf.a1i}))}}}\qquad\left(  \text{by
\eqref{eq.darij2.t4.pf.b1i}}\right) \nonumber\\
&  =g_{r}\left(  a_{p_{1}},a_{p_{2}},\ldots,a_{p_{i-1}},a_{p_{i}}%
a_{n},a_{p_{i+1}},a_{p_{i+2}},\ldots,a_{p_{r}}\right)  .
\label{eq.darij2.t4.pf.bPn.pf.1}%
\end{align}
Now, the recursion \eqref{eq.darij1.2} (applied to $g$, $r+1$ and $\left(
a_{p_{1}},a_{p_{2}},\ldots,a_{p_{r}},a_{n}\right)  $ instead of $f$, $n$ and
$\left(  a_{1},a_{2},\ldots,a_{n}\right)  $) yields
\begin{align*}
&  g_{r+1}\left(  a_{p_{1}},a_{p_{2}},\ldots,a_{p_{r}},a_{n}\right) \\
&  =\underbrace{g\left(  a_{n}\right)  }_{=b_{\left\{  n\right\}  }%
}\underbrace{g_{r}\left(  a_{p_{1}},a_{p_{2}},\ldots,a_{p_{r}}\right)
}_{=b_{P}}\\
&  \qquad-\sum_{i=1}^{r}\underbrace{g_{r}\left(  a_{p_{1}},a_{p_{2}}%
,\ldots,a_{p_{i-1}},a_{p_{i}}a_{n},a_{p_{i+1}},a_{p_{i+2}},\ldots,a_{p_{r}%
}\right)  }_{\substack{=b_{P}^{(p_{i})}\\\text{(by
\eqref{eq.darij2.t4.pf.bPn.pf.1})}}}\\
&  =b_{\left\{  n\right\}  }b_{P}-\sum_{i=1}^{r}b_{P}^{(p_{i})}=b_{\left\{
n\right\}  }b_{P}-\sum_{j\in P}b_{P}^{(j)}%
\end{align*}
(here, we have substituted $j$ for $p_{i}$ in the sum, since $P=\left\{
p_{1},p_{2},\ldots,p_{r}\right\}  $ with $p_{1},p_{2},\ldots,p_{r}$ distinct).
Thus, we can rewrite \eqref{eq.darij2.t4.pf.bPn.pf.0} as
\[
b_{P\cup\left\{  n\right\}  }=\underbrace{b_{\left\{  n\right\}  }b_{P}%
}_{\substack{=b_{P}b_{\left\{  n\right\}  }\\\text{(since }B\text{
is}\\\text{commutative)}}}-\underbrace{\sum_{j\in P}b_{P}^{(j)}}_{=\sum_{i\in
P}b_{P}^{(i)}}=b_{P}b_{\left\{  n\right\}  }-\sum_{i\in P}b_{P}^{(i)}.
\]
This proves \eqref{eq.darij2.t4.pf.bPn}.)

Let $\operatorname*{del}:\Pi_{\left[  n\right]  }\rightarrow\Pi_{\left[
n-1\right]  }$ be the map that transforms each partition $\pi$ of $\left[
n\right]  $ into a partition of $\left[  n-1\right]  $ by deleting the element
$n$ from the block of $\pi$ that contains it (and deleting this block if it
becomes empty). For instance, for $n=4$, we have
\begin{align*}
\operatorname*{del}\left(  \left\{  \left\{  1,4\right\}  ,\left\{
2,3\right\}  \right\}  \right)   &  =\left\{  \left\{  1\right\}  ,\left\{
2,3\right\}  \right\}  \qquad\text{and}\\
\operatorname*{del}\left(  \left\{  \left\{  1\right\}  ,\left\{  2,3\right\}
,\left\{  4\right\}  \right\}  \right)   &  =\left\{  \left\{  1\right\}
,\left\{  2,3\right\}  \right\}  .
\end{align*}

Thus, we can split the sum on the right hand side of Theorem \ref{thm.4} as
follows:
\begin{align}
&  \sum_{\pi=\left\{  P_{1},P_{2},\ldots,P_{k}\right\}  \in\Pi_{\left[
n\right]  }}f_{k}\left(  b_{P_{1}},b_{P_{2}},\ldots,b_{P_{k}}\right)
\nonumber\\
&  =\sum_{\omega=\left\{  Q_{1},Q_{2},\ldots,Q_{\ell}\right\}  \in\Pi_{\left[
n-1\right]  }}\ \ \sum_{\substack{\pi=\left\{  P_{1},P_{2},\ldots
,P_{k}\right\}  \in\Pi_{\left[  n\right]  };\\\operatorname*{del}\left(
\pi\right)  =\omega}}f_{k}\left(  b_{P_{1}},b_{P_{2}},\ldots,b_{P_{k}}\right)
.\ \ \ \ \ \ \ \ \ \ \label{eq.darij2.t4.pf.3}%
\end{align}

Now, fix a partition $\omega=\left\{  Q_{1},Q_{2},\ldots,Q_{\ell}\right\}
\in\Pi_{\left[  n-1\right]  }$.

Which partitions $\pi\in\Pi_{\left[  n\right]  }$ satisfy $\operatorname*{del}%
\left(  \pi\right)  =\omega$ ? In other words, which partitions of $\left[
n\right]  $ become $\omega$ after we remove the element $n$ (and the block
that contains it, in case it becomes empty)? Stated this way, the question is
trivial: those partitions that can be obtained from $\omega$ either by adding
a new block $\left\{  n\right\}  $ or by inserting $n$ into one of the
existing blocks $Q_{1},Q_{2},\ldots,Q_{\ell}$. Thus, there are precisely
$\ell+1$ partitions $\pi\in\Pi_{\left[  n\right]  }$ that satisfy
$\operatorname*{del}\left(  \pi\right)  =\omega$: namely, the one partition
\[
\left\{  Q_{1},Q_{2},\ldots,Q_{\ell},\left\{  n\right\}  \right\}
\]
and the $\ell$ partitions
\[
\left\{  Q_{1},Q_{2},\ldots,Q_{j}\cup\left\{  n\right\}  ,\ldots,Q_{\ell
}\right\}  \qquad\text{for }j\in\left[  \ell\right]
\]
(the notation \textquotedblleft$Q_{1},Q_{2},\ldots,Q_{j}\cup\left\{
n\right\}  ,\ldots,Q_{\ell}$\textquotedblright\ means \textquotedblleft take
the list $Q_{1},Q_{2},\ldots,Q_{\ell}$ and replace its $j$-th entry by
$Q_{j}\cup\left\{  n\right\}  $\textquotedblright). Hence,
\begin{align}
&  \sum_{\substack{\pi=\left\{  P_{1},P_{2},\ldots,P_{k}\right\}  \in
\Pi_{\left[  n\right]  };\\\operatorname*{del}\left(  \pi\right)  =\omega
}}f_{k}\left(  b_{P_{1}},b_{P_{2}},\ldots,b_{P_{k}}\right) \nonumber\\
&  =f_{\ell+1}\left(  b_{Q_{1}},b_{Q_{2}},\ldots,b_{Q_{\ell}},b_{\left\{
n\right\}  }\right) \nonumber\\
&  \qquad+\sum_{j=1}^{\ell}f_{\ell}\left(  b_{Q_{1}},b_{Q_{2}},\ldots
,b_{Q_{j}\cup\left\{  n\right\}  },\ldots,b_{Q_{\ell}}\right)
\label{eq.darij2.t4.pf.5}%
\end{align}
(where \textquotedblleft$b_{Q_{1}},b_{Q_{2}},\ldots,b_{Q_{j}\cup\left\{
n\right\}  },\ldots,b_{Q_{\ell}}$\textquotedblright\ means \textquotedblleft
take the list $b_{Q_{1}},b_{Q_{2}},\ldots,b_{Q_{\ell}}$ and replace its $j$-th
entry by $b_{Q_{j}\cup\left\{  n\right\}  }$\textquotedblright).

Now, let $j\in\left[  \ell\right]  $. Then, $Q_{j}\subseteq\left[  n-1\right]
$ (since $Q_{j}$ is a block of $\omega\in\Pi_{\left[  n-1\right]  }$) and
thus
\[
b_{Q_{j}\cup\left\{  n\right\}  }=b_{Q_{j}}b_{\left\{  n\right\}  }-\sum_{i\in
Q_{j}}b_{Q_{j}}^{(i)}%
\]
(by \eqref{eq.darij2.t4.pf.bPn}, applied to $P=Q_{j}$). Hence,
\begin{align}
&  f_{\ell}\left(  b_{Q_{1}},b_{Q_{2}},\ldots,b_{Q_{j}\cup\left\{  n\right\}
},\ldots,b_{Q_{\ell}}\right) \nonumber\\
&  =f_{\ell}\left(  b_{Q_{1}},b_{Q_{2}},\ldots,b_{Q_{j}}b_{\left\{  n\right\}
}-\sum_{i\in Q_{j}}b_{Q_{j}}^{(i)},\ldots,b_{Q_{\ell}}\right) \nonumber\\
&  =f_{\ell}\left(  b_{Q_{1}},b_{Q_{2}},\ldots,b_{Q_{j}}b_{\left\{  n\right\}
},\ldots,b_{Q_{\ell}}\right) \nonumber\\
&  \qquad-\sum_{i\in Q_{j}}f_{\ell}\left(  b_{Q_{1}},b_{Q_{2}},\ldots
,b_{Q_{j}}^{(i)},\ldots,b_{Q_{\ell}}\right)  \label{eq.darij2.t4.pf.6}%
\end{align}
(since the map $f_{\ell}$ is $\mathbb{Z}$-multilinear and thus, in particular,
linear in its $j$-th argument). Moreover, for each $k\in\left[  \ell\right]
\setminus\left\{  j\right\}  $ and each $i\in Q_{j}$, we have $Q_{k}\cap
Q_{j}=\varnothing$ (since $\omega$ is a set partition, so that its blocks are
disjoint) and thus $i\notin Q_{k}$ (since $i\in Q_{j}$) and therefore
$b_{Q_{k}}=b_{Q_{k}}^{(i)}$ (by \eqref{eq.darij2.t4.pf.bP=bPi}, applied to
$P=Q_{k}$). Hence, in the sum on the right hand side of
\eqref{eq.darij2.t4.pf.6}, we can replace each $b_{Q_{k}}$ by $b_{Q_{k}}%
^{(i)}$. Thus, \eqref{eq.darij2.t4.pf.6} rewrites as
\begin{align}
&  f_{\ell}\left(  b_{Q_{1}},b_{Q_{2}},\ldots,b_{Q_{j}\cup\left\{  n\right\}
},\ldots,b_{Q_{\ell}}\right) \nonumber\\
&  =f_{\ell}\left(  b_{Q_{1}},b_{Q_{2}},\ldots,b_{Q_{j}}b_{\left\{  n\right\}
},\ldots,b_{Q_{\ell}}\right) \nonumber\\
&  \qquad-\sum_{i\in Q_{j}}f_{\ell}\underbrace{\left(  b_{Q_{1}}%
^{(i)},b_{Q_{2}}^{(i)},\ldots,b_{Q_{j}}^{(i)},\ldots,b_{Q_{\ell}}%
^{(i)}\right)  }_{=\left(  b_{Q_{1}}^{(i)},b_{Q_{2}}^{(i)},\ldots,b_{Q_{\ell}%
}^{(i)}\right)  }\nonumber\\
&  =f_{\ell}\left(  b_{Q_{1}},b_{Q_{2}},\ldots,b_{Q_{j}}b_{\left\{  n\right\}
},\ldots,b_{Q_{\ell}}\right) \nonumber\\
&  \qquad-\sum_{i\in Q_{j}}f_{\ell}\left(  b_{Q_{1}}^{(i)},b_{Q_{2}}%
^{(i)},\ldots,b_{Q_{\ell}}^{(i)}\right)  . \label{eq.darij2.t4.pf.7}%
\end{align}

Forget that we fixed $j$. Summing the equality \eqref{eq.darij2.t4.pf.7} over
all $j\in\left[  \ell\right]  $, we find
\begin{align}
&  \sum_{j=1}^{\ell}f_{\ell}\left(  b_{Q_{1}},b_{Q_{2}},\ldots,b_{Q_{j}%
\cup\left\{  n\right\}  },\ldots,b_{Q_{\ell}}\right) \nonumber\\
&  =\sum_{j=1}^{\ell}\Bigg(f_{\ell}\left(  b_{Q_{1}},b_{Q_{2}},\ldots
,b_{Q_{j}}b_{\left\{  n\right\}  },\ldots,b_{Q_{\ell}}\right) \nonumber\\
&  \qquad-\sum_{i\in Q_{j}}f_{\ell}\left(  b_{Q_{1}}^{(i)},b_{Q_{2}}%
^{(i)},\ldots,b_{Q_{\ell}}^{(i)}\right)  \Bigg)\nonumber\\
&  =\sum_{j=1}^{\ell}f_{\ell}\left(  b_{Q_{1}},b_{Q_{2}},\ldots,b_{Q_{j}%
}b_{\left\{  n\right\}  },\ldots,b_{Q_{\ell}}\right) \nonumber\\
&  \qquad-\sum_{j=1}^{\ell}\ \ \sum_{i\in Q_{j}}f_{\ell}\left(  b_{Q_{1}%
}^{(i)},b_{Q_{2}}^{(i)},\ldots,b_{Q_{\ell}}^{(i)}\right)  .
\label{eq.darij2.t4.pf.8}%
\end{align}
In this equality, we can replace the two summation signs $\sum_{j=1}^{\ell
}\ \ \sum_{i\in Q_{j}}$ by $\sum_{i=1}^{n-1}$ (since the sets $Q_{1}%
,Q_{2},\ldots,Q_{\ell}$ form a set partition of $\left[  n-1\right]  $, so
that each $i\in\left[  n-1\right]  $ appears in exactly one of these sets).
Thus, this equality simplifies to
\begin{align}
&  \sum_{j=1}^{\ell}f_{\ell}\left(  b_{Q_{1}},b_{Q_{2}},\ldots,b_{Q_{j}%
\cup\left\{  n\right\}  },\ldots,b_{Q_{\ell}}\right) \nonumber\\
&  =\sum_{j=1}^{\ell}f_{\ell}\left(  b_{Q_{1}},b_{Q_{2}},\ldots,b_{Q_{j}%
}b_{\left\{  n\right\}  },\ldots,b_{Q_{\ell}}\right) \nonumber\\
&  \qquad-\sum_{i=1}^{n-1}f_{\ell}\left(  b_{Q_{1}}^{(i)},b_{Q_{2}}%
^{(i)},\ldots,b_{Q_{\ell}}^{(i)}\right)  . \label{eq.darij2.t4.pf.9}%
\end{align}
Substituting this into \eqref{eq.darij2.t4.pf.5}, we find
\begin{align}
&  \sum_{\substack{\pi=\left\{  P_{1},P_{2},\ldots,P_{k}\right\}  \in
\Pi_{\left[  n\right]  };\\\operatorname*{del}\left(  \pi\right)  =\omega
}}f_{k}\left(  b_{P_{1}},b_{P_{2}},\ldots,b_{P_{k}}\right) \nonumber\\
&  =f_{\ell+1}\left(  b_{Q_{1}},b_{Q_{2}},\ldots,b_{Q_{\ell}},b_{\left\{
n\right\}  }\right) \nonumber\\
&  \qquad+\sum_{j=1}^{\ell}f_{\ell}\left(  b_{Q_{1}},b_{Q_{2}},\ldots
,b_{Q_{j}}b_{\left\{  n\right\}  },\ldots,b_{Q_{\ell}}\right) \nonumber\\
&  \qquad-\sum_{i=1}^{n-1}f_{\ell}\left(  b_{Q_{1}}^{(i)},b_{Q_{2}}%
^{(i)},\ldots,b_{Q_{\ell}}^{(i)}\right)  . \label{eq.darij2.t4.pf.10}%
\end{align}

However, the recursion \eqref{eq.darij1.2} (applied to $\ell+1$ and $\left(
b_{Q_{1}},b_{Q_{2}},\ldots,b_{Q_{\ell}},b_{\left\{  n\right\}  }\right)  $
instead of $n$ and $\left(  a_{1},a_{2},\ldots,a_{n}\right)  $) yields
\begin{align*}
&  f_{\ell+1}\left(  b_{Q_{1}},b_{Q_{2}},\ldots,b_{Q_{\ell}},b_{\left\{
n\right\}  }\right) \\
&  =f\left(  b_{\left\{  n\right\}  }\right)  f_{\ell}\left(  b_{Q_{1}%
},b_{Q_{2}},\ldots,b_{Q_{\ell}}\right) \\
&  \qquad-\sum_{i=1}^{\ell}f_{\ell}\left(  b_{Q_{1}},b_{Q_{2}},\ldots
,b_{Q_{i}}b_{\left\{  n\right\}  },\ldots,b_{Q_{\ell}}\right) \\
&  =f\left(  b_{\left\{  n\right\}  }\right)  f_{\ell}\left(  b_{Q_{1}%
},b_{Q_{2}},\ldots,b_{Q_{\ell}}\right) \\
&  \qquad-\sum_{j=1}^{\ell}f_{\ell}\left(  b_{Q_{1}},b_{Q_{2}},\ldots
,b_{Q_{j}}b_{\left\{  n\right\}  },\ldots,b_{Q_{\ell}}\right)  .
\end{align*}
In other words,
\begin{align*}
&  f_{\ell+1}\left(  b_{Q_{1}},b_{Q_{2}},\ldots,b_{Q_{\ell}},b_{\left\{
n\right\}  }\right) \\
&  \qquad+\sum_{j=1}^{\ell}f_{\ell}\left(  b_{Q_{1}},b_{Q_{2}},\ldots
,b_{Q_{j}}b_{\left\{  n\right\}  },\ldots,b_{Q_{\ell}}\right) \\
&  =f\left(  b_{\left\{  n\right\}  }\right)  f_{\ell}\left(  b_{Q_{1}%
},b_{Q_{2}},\ldots,b_{Q_{\ell}}\right)  .
\end{align*}
Substituting this into \eqref{eq.darij2.t4.pf.10}, we obtain
\begin{align}
&  \sum_{\substack{\pi=\left\{  P_{1},P_{2},\ldots,P_{k}\right\}  \in
\Pi_{\left[  n\right]  };\\\operatorname*{del}\left(  \pi\right)  =\omega
}}f_{k}\left(  b_{P_{1}},b_{P_{2}},\ldots,b_{P_{k}}\right) \nonumber\\
&  =f\left(  b_{\left\{  n\right\}  }\right)  f_{\ell}\left(  b_{Q_{1}%
},b_{Q_{2}},\ldots,b_{Q_{\ell}}\right) \nonumber\\
&  \qquad-\sum_{i=1}^{n-1}f_{\ell}\left(  b_{Q_{1}}^{(i)},b_{Q_{2}}%
^{(i)},\ldots,b_{Q_{\ell}}^{(i)}\right)  . \label{eq.darij2.t4.pf.12}%
\end{align}

Forget that we fixed $\omega$. So we have proved \eqref{eq.darij2.t4.pf.12}
for each partition $\omega=\left\{  Q_{1},Q_{2},\ldots,Q_{\ell}\right\}
\in\Pi_{\left[  n-1\right]  }$. Now, \eqref{eq.darij2.t4.pf.3} becomes
\begin{align*}
&  \sum_{\pi=\left\{  P_{1},P_{2},\ldots,P_{k}\right\}  \in\Pi_{\left[
n\right]  }}f_{k}\left(  b_{P_{1}},b_{P_{2}},\ldots,b_{P_{k}}\right) \\
&  =\sum_{\omega=\left\{  Q_{1},Q_{2},\ldots,Q_{\ell}\right\}  \in\Pi_{\left[
n-1\right]  }}\ \ \sum_{\substack{\pi=\left\{  P_{1},P_{2},\ldots
,P_{k}\right\}  \in\Pi_{\left[  n\right]  };\\\operatorname*{del}\left(
\pi\right)  =\omega}}f_{k}\left(  b_{P_{1}},b_{P_{2}},\ldots,b_{P_{k}}\right)
\\
&  =\sum_{\omega=\left\{  Q_{1},Q_{2},\ldots,Q_{\ell}\right\}  \in\Pi_{\left[
n-1\right]  }}\Bigg(f\left(  b_{\left\{  n\right\}  }\right)  f_{\ell}\left(
b_{Q_{1}},b_{Q_{2}},\ldots,b_{Q_{\ell}}\right) \\
&  \qquad-\sum_{i=1}^{n-1}f_{\ell}\left(  b_{Q_{1}}^{(i)},b_{Q_{2}}%
^{(i)},\ldots,b_{Q_{\ell}}^{(i)}\right)  \Bigg)\qquad\left(  \text{by
\eqref{eq.darij2.t4.pf.12}}\right) \\
&  =f\left(  \underbrace{b_{\left\{  n\right\}  }}_{=g\left(  a_{n}\right)
}\right)  \underbrace{\sum_{\omega=\left\{  Q_{1},Q_{2},\ldots,Q_{\ell
}\right\}  \in\Pi_{\left[  n-1\right]  }}f_{\ell}\left(  b_{Q_{1}},b_{Q_{2}%
},\ldots,b_{Q_{\ell}}\right)  }_{\substack{=\sum_{\pi=\left\{  P_{1}%
,P_{2},\ldots,P_{k}\right\}  \in\Pi_{\left[  n-1\right]  }}f_{k}\left(
b_{P_{1}},b_{P_{2}},\ldots,b_{P_{k}}\right)  \\=h_{n-1}\left(  a_{1}%
,a_{2},\ldots,a_{n-1}\right)  \\\text{(by \eqref{eq.darij2.t4.pf.IH1})}}}\\
&  \qquad-\sum_{i=1}^{n-1}\ \ \underbrace{\sum_{\omega=\left\{  Q_{1}%
,Q_{2},\ldots,Q_{\ell}\right\}  \in\Pi_{\left[  n-1\right]  }}f_{\ell}\left(
b_{Q_{1}}^{(i)},b_{Q_{2}}^{(i)},\ldots,b_{Q_{\ell}}^{(i)}\right)
}_{\substack{=\sum_{\pi=\left\{  P_{1},P_{2},\ldots,P_{k}\right\}  \in
\Pi_{\left[  n-1\right]  }}f_{k}\left(  b_{P_{1}}^{(i)},b_{P_{2}}^{(i)}%
,\ldots,b_{P_{k}}^{(i)}\right)  \\=h_{n-1}\left(  a_{1}^{(i)},a_{2}%
^{(i)},\ldots,a_{n-1}^{(i)}\right)  \\\text{(by \eqref{eq.darij2.t4.pf.IH2})}%
}}\\
&  =\underbrace{f\left(  g\left(  a_{n}\right)  \right)  }%
_{\substack{=h\left(  a_{n}\right)  \\\text{(since }f\circ g=h\text{)}%
}}h_{n-1}\left(  a_{1},a_{2},\ldots,a_{n-1}\right) \\
&  \qquad-\sum_{i=1}^{n-1}h_{n-1}\underbrace{\left(  a_{1}^{(i)},a_{2}%
^{(i)},\ldots,a_{n-1}^{(i)}\right)  }_{\substack{=\left(  a_{1},a_{2}%
,\ldots,a_{i-1},a_{i}a_{n},a_{i+1},a_{i+2},\ldots,a_{n-1}\right)  \\\text{(by
\eqref{eq.darij2.t4.pf.a1i})}}}\\
&  =h\left(  a_{n}\right)  h_{n-1}\left(  a_{1},a_{2},\ldots,a_{n-1}\right) \\
&  \qquad-\sum_{i=1}^{n-1}h_{n-1}\left(  a_{1},a_{2},\ldots,a_{i-1},a_{i}%
a_{n},a_{i+1},a_{i+2},\ldots,a_{n-1}\right)  .
\end{align*}

On the other hand, the recursion \eqref{eq.darij1.2} yields
\begin{align*}
&  h_{n}\left(  a_{1},a_{2},\ldots,a_{n}\right) \\
&  =h\left(  a_{n}\right)  h_{n-1}\left(  a_{1},a_{2},\ldots,a_{n-1}\right) \\
&  \qquad-\sum_{i=1}^{n-1}h_{n-1}\left(  a_{1},a_{2},\ldots,a_{i-1},a_{i}%
a_{n},a_{i+1},a_{i+2},\ldots,a_{n-1}\right)  .
\end{align*}
Comparing these two equalities, we find
\[
h_{n}\left(  a_{1},a_{2},\ldots,a_{n}\right)  =\sum_{\pi=\left\{  P_{1}%
,P_{2},\ldots,P_{k}\right\}  \in\Pi_{\left[  n\right]  }}f_{k}\left(
b_{P_{1}},b_{P_{2}},\ldots,b_{P_{k}}\right)  .
\]
This is precisely Theorem \ref{thm.4} for our $n$. So we have finished the
induction step, and thus the proof of Theorem \ref{thm.4}.
\end{proof}

\subsection{The composition of $n$-homomorphisms}

\begin{proof}
[Proof of Theorem \ref{thm.6}.]Let $h:=f\circ g:A\rightarrow C$. Thus, we must
show that $h$ is an $nm$-homomorphism. In other words, we must show that $h$
is central and satisfies $h_{nm+1}=0$.

Since $g$ is central, it is easy to see that $h$ is central as well (indeed,
all $a,a^{\prime}\in A$ satisfy
\begin{align*}
h\left(  aa^{\prime}\right)   &  =f\left(  g\left(  aa^{\prime}\right)
\right)  \qquad\left(  \text{since }h=f\circ g\right) \\
&  =f\left(  g\left(  a^{\prime}a\right)  \right)  \qquad\left(  \text{since
}g\text{ is central, so }g\left(  aa^{\prime}\right)  =g\left(  a^{\prime
}a\right)  \right) \\
&  =h\left(  a^{\prime}a\right)  \qquad\left(  \text{since }h=f\circ g\right)
,
\end{align*}
which shows that $h$ is central). Thus, it remains to prove that $h_{nm+1}=0$.

Set $N:=nm+1$. Let $a_{1},a_{2},\ldots,a_{N}\in A$. We shall prove that
$h_{N}\left(  a_{1},a_{2},\ldots,a_{N}\right)  =0$.

For any subset $P=\left\{  p_{1},p_{2},\ldots,p_{r}\right\}  $ of $\left[
N\right]  $ (with $p_{1},p_{2},\ldots,p_{r}$ distinct), we set
\[
b_{P}:=g_{r}\left(  a_{p_{1}},a_{p_{2}},\ldots,a_{p_{r}}\right)  .
\]
(This is well-defined for the same reason as in Theorem \ref{thm.4}.) Then,
Theorem \ref{thm.4} (applied to $N$ instead of $n$) shows that
\begin{equation}
h_{N}\left(  a_{1},a_{2},\ldots,a_{N}\right)  =\sum_{\pi=\left\{  P_{1}%
,P_{2},\ldots,P_{k}\right\}  \in\Pi_{\left[  N\right]  }}f_{k}\left(
b_{P_{1}},b_{P_{2}},\ldots,b_{P_{k}}\right)  \label{eq.darij2.t6.pf.1}%
\end{equation}
(where $f_{k}\left(  b_{P_{1}},b_{P_{2}},\ldots,b_{P_{k}}\right)  $ is
well-defined for the same reason as in Theorem \ref{thm.4}).

However, we claim that any partition $\pi=\left\{  P_{1},P_{2},\ldots
,P_{k}\right\}  \in\Pi_{\left[  N\right]  }$ satisfies
\begin{equation}
f_{k}\left(  b_{P_{1}},b_{P_{2}},\ldots,b_{P_{k}}\right)  =0.
\label{eq.darij2.t6.pf.2}%
\end{equation}

[\textit{Proof:} Let $\pi=\left\{  P_{1},P_{2},\ldots,P_{k}\right\}  \in
\Pi_{\left[  N\right]  }$ be any partition. Since $f$ is an $n$-homomorphism,
we have $f_{n+1}=0$. Thus, if $k\geq n+1$, then $f_{k}=0$ as well (by Lemma
\ref{lem.5}, applied to $B$, $C$, $n+1$ and $k$ instead of $A$, $B$, $m$ and
$n$), and therefore \eqref{eq.darij2.t6.pf.2} certainly holds in this case.
Hence, for the rest of this proof of \eqref{eq.darij2.t6.pf.2}, we WLOG assume
that $k<n+1$. Therefore, $k\leq n$, so that $n\geq k$. Now, since $\left\{
P_{1},P_{2},\ldots,P_{k}\right\}  $ is a partition of $\left[  N\right]  $, we
have
\[
\left\vert P_{1}\right\vert +\left\vert P_{2}\right\vert +\cdots+\left\vert
P_{k}\right\vert =\left\vert \left[  N\right]  \right\vert
=N=nm+1>\underbrace{n}_{\geq k}m\geq km.
\]
Hence, at least one of the $k$ addends $\left\vert P_{1}\right\vert
,\left\vert P_{2}\right\vert ,\ldots,\left\vert P_{k}\right\vert $ must be
larger than $m$ (because otherwise, we would have $\left\vert P_{1}\right\vert
\leq m$ and $\left\vert P_{2}\right\vert \leq m$ and $\ldots$ and $\left\vert
P_{k}\right\vert \leq m$, and therefore, adding these $k$ inequalities
together, we would obtain $\left\vert P_{1}\right\vert +\left\vert
P_{2}\right\vert +\cdots+\left\vert P_{k}\right\vert \leq
\underbrace{m+m+\cdots+m}_{k\text{ times}}=km$, which would contradict
$\left\vert P_{1}\right\vert +\left\vert P_{2}\right\vert +\cdots+\left\vert
P_{k}\right\vert >km$). In other words, there exists some $j\in\left[
k\right]  $ such that $\left\vert P_{j}\right\vert >m$. Consider this $j$.
Now, recall that $g$ is an $m$-homomorphism; thus, $g_{m+1}=0$. Write the set
$P_{j}$ as $P_{j}=\left\{  p_{1},p_{2},\ldots,p_{r}\right\}  $ (with
$p_{1},p_{2},\ldots,p_{r}$ distinct); thus, $r=\left\vert P_{j}\right\vert
>m$. Therefore, $r\geq m+1$. Hence, from $g_{m+1}=0$, we obtain $g_{r}=0$ (by
Lemma \ref{lem.5}, applied to $g$, $m+1$ and $r$ instead of $f$, $m$ and $n$).

But the definition of $b_{P_{j}}$ yields $b_{P_{j}}=g_{r}\left(  a_{p_{1}%
},a_{p_{2}},\ldots,a_{p_{r}}\right)  =0$ (since $g_{r}=0$). Thus,
$f_{k}\left(  b_{P_{1}},b_{P_{2}},\ldots,b_{P_{k}}\right)  =0$ as well (since
$f_{k}$ is a $\mathbb{Z}$-multilinear map, and thus vanishes if any of the $k$
inputs is $0$). This proves \eqref{eq.darij2.t6.pf.2}.] \medskip

Now, substituting \eqref{eq.darij2.t6.pf.2} into \eqref{eq.darij2.t6.pf.1}, we
find
\[
h_{N}\left(  a_{1},a_{2},\ldots,a_{N}\right)  =\sum_{\pi=\left\{  P_{1}%
,P_{2},\ldots,P_{k}\right\}  \in\Pi_{\left[  N\right]  }}0=0.
\]
Since we have proved this for all $a_{1},a_{2},\ldots,a_{N}\in A$, we thus
obtain $h_{N}=0$. In other words, $h_{nm+1}=0$ (since $N=nm+1$). As we said,
this completes the proof of Theorem \ref{thm.6}.
\end{proof}

\end{document}